\documentclass[a4paper,english,sumlimits,reqno]{amsart}

\usepackage[latin1]{inputenc}
\usepackage[T1]{fontenc}
\usepackage{babel}

\usepackage{graphicx}

\usepackage[section]{placeins}

\usepackage{enumerate}

\usepackage{varioref}

\usepackage{amssymb}
\usepackage{amsmath}
\usepackage{amsthm}
\usepackage{mathtools}

\usepackage[mathcal]{eucal}
\usepackage{mathrsfs}

\theoremstyle{plain}

\newtheorem{thm}{Theorem}[section]
\newtheorem*{thm*}{Theorem}

\newtheorem{pro}[thm]{Proposition}

\newtheorem{lem}[thm]{Lemma}
\theoremstyle{definition}

\theoremstyle{remark}

\newtheorem*{rem}{Remark}

\numberwithin{equation}{section}

\newcommand{\bb}[1]{\ensuremath{\mathbb #1}}
\newcommand{\cal}[1]{\ensuremath{\mathcal #1}}
\newcommand{\scr}[1]{\ensuremath{\mathscr #1}}

\newcommand{\embed}{\hookrightarrow}
\newcommand{\union}{\cup}
\newcommand{\Union}{\bigcup}
\newcommand{\isect}{\cap}
\newcommand{\Isect}{\bigcap}

\newcommand{\umd}{\ensuremath{\mathrm{UMD}}}
\newcommand{\umdconst}{\ensuremath{\mathscr{U}}}

\newcommand{\charfun}{\ensuremath{\mathtt 1}}
\newcommand{\dcubes}{\scr Q}
\newcommand{\dint}{\scr D}
\newcommand{\lip}{\ensuremath{\mathrm{Lip}}}

\newcommand{\gint}[3][\ ]{\ensuremath{\int_{#1} #3 \, d#2}}

\DeclareMathOperator{\card}{\#}

\DeclareMathOperator{\supp}{supp}
\DeclareMathOperator{\dist}{dist}
\DeclareMathOperator{\diam}{diam}
\DeclareMathOperator{\sidelength}{sl}

\DeclareMathOperator{\lin}{lin}

\DeclareMathOperator{\cond}{\bb E}

\DeclareMathOperator{\salg}{\sigma}
\DeclareMathOperator{\pred}{\pi}

\title[Interpolatory Estimate for $\umd$--Valued Directional Haar Projection]{An Interpolatory Estimate for the $\umd$--Valued Directional Haar Projection}

\author{Richard Lechner}

\thanks{Supported by FWF P20166-N18}
\thanks{\noindent This is part of my PhD thesis written at Department of
  Analysis, J. Kepler University Linz. I want to thank my advisor
  P. F. X. M\"uller for many helpful discussions during the
  preparation of this thesis.}

\address{Department of Analysis
  \newline
  \indent Johannes Kepler Universitaet Linz
  \newline
  \indent Altenbergerstrasse 69
  \newline
  \indent A-4040 Linz, Austria
}

\date{\today}

\begin{document}

\maketitle

\tableofcontents

\newpage
\section{Main Results}\label{s:results}

\subsection{A Brief History of Developement}\label{ss:history_dev}\hfill

The Calculus of Variations, in particular the theory of compensated
compactness has long been a source of hard problems in harmonic
analysis.
One developement started with the work of F. Murat and L. Tartar and
especially in the papers of Murat
(\cite{tartar:1978,tartar:1979,tartar:1983,tartar:1984,tartar:1990,tartar:1993},
and~\cite{murat:1978,murat:1979,murat:1981}).
The decicive theorems were on Fourier multipliers of H\"ormander
type. For extensions of the use of Fourier multipliers in relation to
sequential weak lower semicontinuity of integrals of the form
\begin{equation*}
  (u,v) \mapsto \int f(x,u(x),v(x))\, \mathrm d x,
\end{equation*}
and Young Measures and a full developement of the method
see~\cite{fonseca_mueller:1999}.
The extensions are due to S. Mueller
(see \cite{s_mueller:1999}), who used time--frequency localization and
modern Calderon--Zygmund theory to strengthen the results obtained by
Fourier multiplier methods.

Let $u \in L^p(\bb R^n)$, with $n \geq 2$ and $1 < p < \infty$ fixed,
then the directional Haar projection
$ P^{(\varepsilon)}\, :\, L^p(\bb R^n) \longrightarrow L^p(\bb R^n)$,
is given by
\begin{equation*}
  P^{(\varepsilon)} u
  = \sum_{Q \in \dcubes}
    \langle u, h_Q^{(\varepsilon)} \rangle\,
    h_Q^{(\varepsilon)}\, |Q|^{-1}.
\end{equation*}
For a precise definition see~\eqref{eqn:directional_projection}.
In~\cite{s_mueller:1999} S. Mueller obtained the result
\begin{equation}\label{eqn:intro:result_mue99}
  \|P^{(\varepsilon)} u \|_{L^2(\bb R^2)}
  \leq C\, \|u\|_{L^2(\bb R^2)}^{1/2}\,
    \|R_{i_0} u\|_{L^2(\bb R^2)}^{1-1/2},
\end{equation}
where $R_{i_0}$ denotes the $i_0$--th Riesz transform in $\bb R^2$,
$0 \neq (\varepsilon_1,\varepsilon_2) = \varepsilon \in \{0,1\}^2$,
and $\varepsilon_{i_0} = 1$.
The formal definition of the Riesz transform is supplied in
section~\ref{s:preliminaries}.

This inequality was then extended by J. Lee, P. F. X. Mueller and
S. Mueller in~\cite{lee_mueller_mueller:2007} to arbitrary
$1 < p < \infty$ and dimensions $n \geq 2$ to
\begin{equation}\label{eqn:intro:result_lmm07}
  \|P^{(\varepsilon)} u \|_{L^p(\bb R^n)}
  \leq C\, \|u\|_{L^p(\bb R^n)}^{1/\min(2,p)}\,
    \|R_{i_0} u\|_{L^p(\bb R^n)}^{1-1/\min(2,p)},
\end{equation}
where $\varepsilon \in \{0,1\}^n \setminus \{0\}$,
$\varepsilon_{i_0} = 1$.
Note that the behaviour of this inequality for $1 < p < 2$ and
$2 < p < \infty$ strongly varies.
The most important application of~\eqref{eqn:intro:result_lmm07}
appears for $p = n$.
One can rewrite~\eqref{eqn:intro:result_lmm07} using the notion of
type $\cal T(L^p(\bb R^n)) = \min(2,p)$
\begin{equation*}
  \|P^{(\varepsilon)} u \|_{L^p(\bb R^n)}
  \leq C\, \|u\|_{L^p(\bb R^n)}^{1/\cal T(L^p(\bb R^n))}\,
  \|R_{i_0} u\||_{L^p(\bb R^n)}^{1-1/\cal T(L^p(\bb R^n))}.
\end{equation*}
The proofs of~\eqref{eqn:intro:result_mue99} as well
as~\eqref{eqn:intro:result_lmm07} are based on two consecutive and ad
hoc defined time--frequency localizations of the operator
$P^{(\varepsilon)}$, based on Littlewood--Paley and wavelet
expansions.

\subsection{The Main Result}\label{ss:main_result}\hfill

S. Mueller asks in~\cite{s_mueller:1999} whether it is possible to
obtain~\eqref{eqn:intro:result_mue99} in such a way that the original
time--frequency decompositions are replaced by the \textbf{canonical
  martingale decomposition} of T. Figiel (see~\cite{figiel:1988}
and~\cite{figiel:1991}).
This paper provides an affirmative answer to this question, and
thus extending the interpolatory
estimate~\eqref{eqn:intro:result_lmm07} to the Bochner--Lebesgue space
$L_X^p(\bb R^n)$, provided $X$ satisfies the $\umd$--property.

Our methods are based on martingale methods, explaining the behaviour
of the exponents in the following main
inequality~\ref{eqn:main_result} in terms of type and cotype.
The main result of this paper reads as follows.
\begin{thm*}[Main Result]\label{thm:main_result}
  Let $1 < p < \infty$, $1 \leq i_0 \leq n$ and
  $\varepsilon = (\varepsilon_1,\ldots,\varepsilon_n) \in \{0,1\}^n$
  such that $\varepsilon_{i_0} = 1$.
  If $X$ has the $\umd$--property and $L_X^p$ has non--trivial type
  $\cal T(L_X^p)$, then there exists a constant $C$, such that for all
  $u \in L_X^p$
  \begin{equation}\label{eqn:main_result}
    \|P^{(\varepsilon)} u\|_{L_X^p}
    \leq C\ \|u\|_{L_X^p}^{1/\cal T(L_X^p)}\,
      \|R_{i_0} u\|_{L_X^p}^{1 - 1/\cal T(L_X^p)},
  \end{equation}
  whereas the constant $C$ depends only on $n$, $p$, $X$ and
  $\cal T(L_X^p)$.
\end{thm*}

\begin{proof}
  First define $M \in \bb N$ by
  \begin{equation}
    2^{M-1}
    \leq \frac{\|R_{i_0}\|_{L_X^p,L_X^p}\, \|u\|_{L_X^p}}
      {\|R_{i_0}u\|_{L_X^p}}
    \leq 2^M.
    \label{eqn:main_thm--splitting_index}
  \end{equation}
  After using
  decomposition~\eqref{eqn:directional_projection--decomposition}, the
  triangle inequality, 
  estimates~\eqref{eqn:mollified_operator_riesz_estimate--l<=0},
  \eqref{eqn:mollified_operator_estimate--l>=0-1},
  \eqref{eqn:mollified_operator_riesz_estimate--l>=0} and plugging in
  $M$, we obtain
  \begin{align*}
    \|P^{(\varepsilon)} u\|_{L_X^p}
    & \leq \| P_-^{(\varepsilon)} R_{i_0}^{-1} R_{i_0} u\|_{L_X^p}
      + \sum_{l = 0}^M \|P_l R_{i_0}^{-1} R_{i_0} u\|_{L_X^p}
      + \sum_{l = M}^\infty \|P_l u\|_{L_X^p}\\
    & \lesssim \|R_{i_0} u\|_{L_X^p}
      + \sum_{l = 0}^M 2^{l/\cal T(L_X^p)}\, \|R_{i_0} u\|_{L_X^p}
      + \sum_{l = M}^\infty
        2^{-l (1-\frac{1}{\cal T(L_X^p)})}\, \|u\|_{L_X^p}\\
    & \lesssim 2^{M/\cal T(L_X^p)}\, \|R_{i_0} u\|_{L_X^p}
      + 2^{-M (1-\frac{1}{\cal T(L_X^p)})}\,
      \|u\|_{L_X^p}\\
    & \leq C\, \|u\|_{L_X^p}^{1/\cal T(L_X^p)}
      \|R_{i_0} u\|_{L_X^p}^{1 - 1/\cal T(L_X^p)}.
  \end{align*}
\end{proof}

The basic tools for the proof of the above theorem are vector--valued
estimates of so called ring domain operators, developed
in section~\ref{s:ring_domain_operator}.
A careful examination of T. Figiel's shift operators acting on ring
domains will be crucial in those estimates.

\bigskip

Clearly, the main result, theorem~\ref{thm:main_result}, represents a
result on interpolation of operators, linking the identity map, the
Riesz transforms and the directional Haar projection. We would now
like to give a reformulation of our main theorem which places it in
the context of structure theorems for the so called $K$--method of
interpolation spaces.
To this end, we first introduce the $K$--functional, cite the relevant
structure theorem and apply it to the inequalities stated as our main
result. 

Define the $K$--functional
\begin{equation*}
  K(f,t)
  = \inf \big\{\|g\|_{E_0} + t\, \|h\|_{E_1}\, :\,
    f = g + h,\ g \in E_0, h \in E_1
  \big\},
\end{equation*}
for all $f \in E_0 + E_1$ and $t > 0$,
and the interpolation space
\begin{equation*}
  (E_0,E_1)_{\theta,1}
  = \big\{ f\, :\,
    f \in E_0 + E_1,\ 
    \|f\|_{\theta,1} < \infty
  \big\},
\end{equation*}
where
\begin{equation*}
  \|f\|_{\theta,1} = \int_0^\infty t^{-\theta}\, K(f,t)\,
    \frac{\mathrm d t}{t}.
\end{equation*}

The following proposition interprets interpolatory estimates such as
the ones obtained in our main theorem in terms of continuity of
the identity map between interpolation spaces.
The following proposition is a result of general interpolation theory
(see~\cite[Proposition 2.10, Chapter 5]{bennett_sharpley:1988}).
\begin{pro}\label{pro:interpolation_general_result}
  Let $(E_0,E_1)$ be a compatible couple and suppose $0 < \theta < 1$.
  Then the estimate
  \begin{equation}
    \|f\|_E \leq C\, \|f\|_{\theta,1} 
  \end{equation}
  holds for some constant $C$ and all $f$ in $(E_0,E_1)_{\theta,1}$ if
  and only if
  \begin{equation*}
    \|f\|_E \leq C\, \|f\|_{E_0}^{1-\theta}\, \|f\|_{E_1}^\theta
  \end{equation*}
  holds for some constant $C$ and for all $f$ in $E_0 \isect E_1$.  
\end{pro}

Now we specify how to choose the spaces $E$, $E_0$ and $E_1$ so that
the two equivalent conditions of
the above proposition match precisely the assertions of our main
theorem, see inequality~\eqref{eqn:main_result}.

Fix $0 \neq \varepsilon \in \bb \{0,1\}^n$, let
$R$ denote one of the Riesz transform operators
\begin{equation*}
  R_i\, :\, L_X^p \rightarrow L_X^p
\end{equation*}
defined in section~\ref{s:preliminaries}, where $\varepsilon_i = 1$,
and abbreviate $P^{(\varepsilon)}$ by $P$.
If we define the Banach spaces
\begin{align*}
  E & = L_X^p / \ker(P),
  & \|u + \ker(P)\|_E & = \|P u\|_{L_X^p},\\
  E_0 & = L_X^p,
  & \|u\|_{E_0} & = \|u\|_{L_X^p},\\
  E_1 & = L_X^p / \ker(R),
  & \|u + \ker(R)\|_{E_1} & = \|R u\|_{L_X^p},
\end{align*}
then in view of proposition~\ref{pro:interpolation_general_result}
\begin{equation*}
  (E_0,E_1)_{\theta,1} \embed E,
\end{equation*}
is equivalent to the existence of a constant $C > 0$ such that
\begin{equation*}
  \|u\|_E \leq C\, \|u\|_{\theta,1},
\end{equation*}
for all $u \in (E_0,E_1)_{\theta,1}$.

\bigskip

We are grateful to S. Geiss who pointed out the connection to general
interpolation theory.

\newpage
\section{Preliminaries}\label{s:preliminaries}

This brief section will provide notions and tools most frequently used
in what follows.

At first we will introduce the Haar system supported on dyadic cubes,
the notions of Banach spaces with the $\umd$--property and type and
cotype of Banach spaces.
The $\umd$--property enables us to introduce Rademacher means in our
norm estimates, so that we may use the subsequent inequalities, that
is Kahane's inequality, Kahane's contraction principle and Bourgain's
version of Stein's martingale inequality.

Then we turn to Figiel's shift operators $T_m$ acting on all of the
Haar system, where $T_m$ is bounded by a constant multiple of
\begin{equation*}
  \log (2 + |m|), \qquad m \in \bb Z^n.
\end{equation*}
Very roughly speaking this result due to T. Figiel is obtained by
partitioning all of the dyadic cubes into $\log (2 + |m|)$
collections and bounding $T_m$ by a constant on each of the
collections.
We will have to consider $T_m$ acting only in one direction (assume
$m_2 = \ldots = m_n = 0$) on the Haar spectrum of certain ring domain
operators $S_\lambda$, $\lambda \geq 0$, and it turns out that $T_m$
restricted to this spectrum is uniformly bounded by a constant, as
long as $0 \leq m_1 \leq 2^\lambda -1$.

\subsubsection*{The Haar System}\hfill

First we consider the collection of dyadic intervals at scale
$j \in \bb Z$
\begin{equation*}
  \dint_j
  = \big\{\ [2^{-j} k, 2^{-j}(k+1) [\, :\, k \in \bb Z \ \big\},
\end{equation*}
and the collection of all dyadic intervals
\begin{equation*}
  \dint = \Union_{j \in \bb Z} \dint_j.
\end{equation*}
Now define the $L^\infty$--normalized Haar system
\begin{align*}
  h_{[0,1[}(t)
  & = \charfun_{[0,\frac{1}{2}[}(t) - \charfun_{[\frac{1}{2},1[}(t),
  & t & \in \bb R
  \intertext{and for any $I \in \dint$}
  h_I(t) & = h_{[0,1[}\big( \frac{t - \inf I}{|I|} \big),
  & t & \in \bb R.
\end{align*}

\smallskip

In arbitrary dimensions $n \geq 2$ one can obtain a basis for
$L^p(\bb R^n)$ as follows.
For any
$\varepsilon = (\varepsilon_1,\ldots,\varepsilon_n) \in \{0,1\}^n$,
$\varepsilon \neq 0$ define
\begin{equation*}
  h^{(\varepsilon)}_Q(t) = \prod_{i=1}^n h_{I_i}^{\varepsilon_i}(t_i),
\end{equation*}
where $t = (t_1,\ldots,t_n) \in \bb R^n$,
$Q = I_1 \times \cdots \times I_n$,
$|I_1| = \ldots = |I_n|$, $I_i \in \dint$, and by
$h_{I_i}^{\varepsilon_i}$ we mean
\begin{equation*}
  h_{I_i}^{\varepsilon_i}
  =
  \begin{cases}
    h_{I_i} & \varepsilon_i=1\\
    \charfun_{I_i} & \varepsilon_i=0
  \end{cases}
\end{equation*}
Note that the former basis is supported on rectangles $R$, but the
latter basis is supported on dyadic cubes $Q$.

\subsubsection*{Banach Spaces with the $ \umd$--Property}\hfill

By $L^p(\Omega, \mu; X)$ we denote the space of functions with values
in $X$, Bochner--integrable with respect to $\mu$.
If $\Omega = \bb R^n$ and $\mu$ is the Lebesgue measure $|\cdot|$ on $\bb R^n$,
then set $L_X^p(\bb R^n) = L^p(\bb R^n, |\cdot|; X)$, if unambiguous
abbreviated as $L_X^p$.

\smallskip

We say $X$ is a $\umd$ space if for any $X$--valued martingale
difference sequence $\{d_j\}_j \subset L^p(\Omega,\mu;X)$ and any
choice of signs $\varepsilon_j \in \{-1,1\}$ one has
\begin{equation}\label{eqn:umd--property}
  \big\| \sum_j \varepsilon_j\, d_j \big\|_{L^p(\Omega,\mu;X)}
  \leq \umdconst_p(X)\, \big\| \sum_j d_j \big\|_{L^p(\Omega,\mu;X)}.
\end{equation}

A Banach space $X$ is said to be of type $\cal T$,
$1 < \cal T \leq 2$, respectively of cotype $\cal C$,
$2 \leq \cal C < \infty$
if there are constansts $A(\cal T,X) > 0$ and $B(\cal C,X) > 0$, such
that for every finite set of vectors $\{x_j\}_j \subset X$ we have
\begin{align}
  \int_0^1 \big\| \sum_j r_j(t)\, x_j \big\|_X\, \mathrm d t
  & \leq A(\cal T,X)\,
    \big( \sum_j \| x_j \|_X^{\cal T} \big)^{1/\cal T},
    \label{eqn:type}
  \intertext{respectively}
  \int_0^1 \big\| \sum_j r_j(t)\, x_j \big\|_X\, \mathrm d t
  & \geq B(\cal C,X)\,
    \big( \sum_j \| x_j \|_X^{\cal C} \big)^{1/\cal C},
    \label{eqn:cotype}
\end{align}
where $\{r_j\}_j$ is an independent sequence of Rademacher functions.

It is well known that if $X$ is a $\umd$--space, then for every
$1 < p < \infty$ the Lebesgue--Bochner space $L_X^p$ has
(non--trivial) type and cotype.
Since inequality~\eqref{eqn:type} holds for $\cal T = 1$, respectively
inequality~\eqref{eqn:cotype} with $\cal C = \infty$, even if $X$ does
not have the $\umd$--property, one often refers to $\cal T = 1$ as
trivial type, respectively to $\cal C = \infty$ as trivial cotype.

\subsubsection*{Kahane's Inequality}\hfill

Given $1 \leq p < \infty$, there exists a constant $K_p$ such that for
any Banach space $X$ and any finite sequence $\{x_j\} \subset X$ holds
that
\begin{equation}\label{eqn:kahanes_inequality}
  \bigg(
    \int_0^1 \Big\| \sum_j r_j(t)\, x_j \Big\|_X^p\, \mathrm d t
  \bigg)^{1/p}
  \leq K_p\,
  \int_0^1 \Big\| \sum_j r_j(t)\, x_j \Big\|_X\, \mathrm d t,
\end{equation}
where $\{r_j\}_j$ denotes an independent sequence of Rademacher
functions.

\subsubsection*{Kahane's Contraction Principle}\hfill

For any Banach space $X$, $1 < p < \infty$, finite set
$\{x_j\} \subset X$ and bounded sequence of scalars $\{c_j\}$ holds
true
\begin{equation}\label{eqn:kahanes_contraction_principle}
  \int_0^1 \Big\| \sum_j r_j(t)\, c_j\, x_j \Big\|_X^p\, \mathrm d t
  \leq \sup_j |c_j|^p\,
  \int_0^1 \Big\| \sum_j r_j(t)\, x_j \Big\|_X^p\, \mathrm d t,
\end{equation}
where $\{r_j\}_j$ denotes an independent sequence of Rademacher
functions.

\subsubsection*{The Martingale Inequality of Stein -- Bourgain's Version}\hfill

The vector--valued version of Stein's martingale inequality states
that if $(\Omega,\cal F,\mu)$ is a probability space,
$\cal F_1 \subset \ldots \subset \cal F_m \subset \cal F$ is an
increasing sequence of $\sigma$--algebras,
$f_1,\ldots,f_m \in L^p(\Omega,\mu;X)$ and $r_1,\ldots,r_m$ are
independent Rademacher functions, then
\begin{equation}\label{eqn:steins_martinagle_inequality}
  \int_0^1 \big\|
    \sum_{i=1}^m r_i(t)\, \cond(f_i | \cal F_i)
  \big\|_{L^p(\Omega,\mu;X)}\, \mathrm dt
  \leq C\,   \int_0^1 \big\|
    \sum_{i=1}^m r_i(t)\, f_i
  \big\|_{L^p(\Omega,\mu;X)}\, \mathrm dt,
\end{equation}
where $C$ depends on $p$ and $X$.
The Banach space $X$ having the $\umd$--property assures
$C < \infty$.

\bigskip

\subsubsection*{Figiel's Shift Operators}\hfill

The proof of the main result~\eqref{eqn:main_result} makes use
of Figiel's shift operators \cite{figiel:1988}.
For any $m \in \bb Z^n$, and collection $\scr B$ of cubes
$Q \subset \bb R^n$ let $\tau_m\, :\, \scr B \rightarrow \dcubes$
\begin{equation}\label{eqn:shift_function}
  \tau_m(Q) = Q + m\, \sidelength(Q),
\end{equation}
where $\sidelength(Q)$ is the sidelength of $Q$.
Precisely, if $Q = I_1\times \ldots \times I_n$, with
$|I_1| = \ldots = |I_n|$, then
$\sidelength(Q) = |I_1| = \ldots = |I_n|$.

The map $\tau_m$ induces the rearrangement operator $T_m$, as the
linear extension of
\begin{equation}\label{eqn:shift_operator_1}
  T_m h_Q = h_{\tau_m(Q)}, \qquad Q \in \scr B.
\end{equation}

Let $X$ be a $\umd$ space, then the theorem of T. Figiel bounds the
shift operator $T_m$ acting on $L_X^p$ by
\begin{equation}\label{eqn:shift_operator_1_estimate}
  \| T_m\, :\, L_X^p \rightarrow L_X^p \|
  \leq C\, \log(2 + |m|),
\end{equation}
where $C = C(n,p,\umdconst_p(X))$.

\subsubsection*{The Riesz Transform}\hfill

Formally, we define the Riesz transform $R_{i_0}$ by
\begin{align}\label{eqn:riesz_transform}
  R_{i_0} f & = K_{i_0} * f,\\
  K_{i_0}(x) & = c_n\, \frac{x_{i_0}}{|x|^{n+1}},
  & x & = (x_1, \ldots,x_n).
\end{align}
Details may be found in~\cite{stein:1970} and~\cite{stein:1993}.

\subsubsection*{Supplementary Definitions}\hfill

\begin{equation*}
  \cal E
  = \cal E(n)
  = \big\{
    \varepsilon \in \{0,1\}^n\, :\,  \varepsilon \neq (0,\ldots,0)
  \big\},
\end{equation*}
and 
\begin{equation*}
  \cal E_{i_0}
  = \cal E_{i_0}(n)
  = \big\{
    \varepsilon \in \{0,1\}^n\, :\,  \varepsilon_{i_0} \neq 0
  \big\}.
\end{equation*}

For any operator $T:L_X^p \rightarrow L_X^p$, the Haar--spectrum is
defined by
\begin{equation}\label{eqn:haar--support}
  \dcubes \setminus
  \big\{ Q \in \dcubes\, :\,
  \langle T u, h_Q^{(\varepsilon)} \rangle = 0,
  \text{for all $u\in L_X^p$ and $\varepsilon \in \cal E$}
  \big\}.
\end{equation}

\bigskip

Given a collection of sets $\scr C$, we denote
\begin{equation*}
  \sigma(\scr C)
  = \Isect \big\{
    \scr A\, :\, \text{$\scr A$ is a $\sigma$--algebra},
    \scr C \subset \scr A
  \big\},
\end{equation*}
the smallest $\sigma$--algebra containing $\scr C$.

\newpage
\section{The Ring Domain Operator $S_\lambda$}\label{s:ring_domain_operator}

Here we define and study the ring domain operators $S_\lambda$,
mapping $u \in L_X^p$ onto the blocks
$\lin \big\{g_{Q,\lambda}\, :\, Q \in \dcubes \big\}$, each supported
on a ring--shaped structure, see figure~\vref{pic:ring_domain}, from
now on referred to as ring domain.
The vector--valued estimates for these operators constitute the
technical main component of this paper.

The main result for the ring domain operator is stated in
theorem~\vref{thm:ring_domain_operator}.

\subsection{Preparation}\label{ss:ring_domain_operator-preparation}\hfill

Now we turn to defining ring domains and their corresponding ring
domain operators.
Within this section the superscripts $\varepsilon$ are omitted, we
assume $\lambda \geq 0$ and generically denote $h_Q$ one of the
functions $\{h_Q^{(\varepsilon)}\}_{\varepsilon \neq 0}$.

Let $D(Q)$ be the set of discontinuities of the Haar function $h_Q$,
then
\begin{equation*}
  D_\lambda(Q)
  = \{ x \in \bb R^n\, :\,
    \dist(x, D(Q)) \leq C\, 2^{-\lambda}\, \diam(Q) \}.
\end{equation*}

First note that
\begin{equation}
  |D_\lambda(Q)| \lesssim 2^{-\lambda}\, |Q|.
\end{equation}
Now we cover the set $D_\lambda(Q)$ using dyadic cubes
$E(Q)$ having diameter
\begin{equation*}
  \diam(E(Q)) = 2^{-\lambda}\, \diam(Q),
\end{equation*}
and call the collection of those cubes $\cal U_\lambda(Q)$.
More precisely,
\begin{equation}\label{eqn:dyadic_support}
  \cal U_\lambda(Q)
  = \big\{ E \in \dcubes\, :\, \diam(E) = 2^{-\lambda}
    \diam(Q),\,  E \isect D_\lambda(Q) \neq \emptyset \big\}.
\end{equation}
The pointset $U_\lambda(Q)$ covered by $\cal U_\lambda(Q)$
is illustrated by the shaded region in figure~\vref{pic:ring_domain},
wherein the dashed lines represent the set of discontinuities $D(Q)$.
\begin{figure}[bt]
  \centering
  \includegraphics{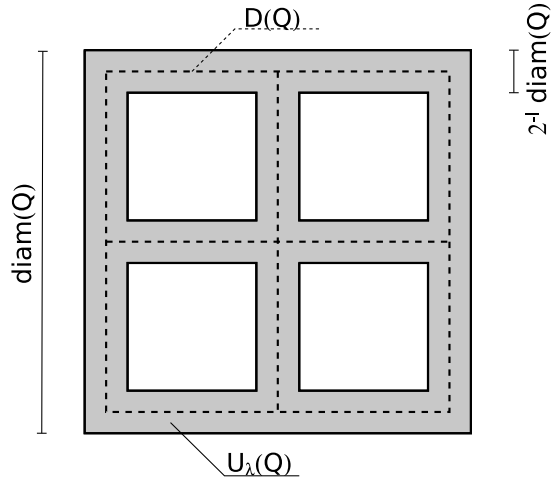}
  \caption{Ring domain (shaded region).}
  \label{pic:ring_domain}
\end{figure}

The cardinality $\card \cal U_\lambda(Q)$ does not depend on the
choice of $Q$, precisely
\begin{equation}
  \card \cal U_\lambda(Q) \approx 2^{\lambda(n-1)}.
\end{equation}
Now we define the functions $g_{Q,\lambda}$
associated to the ring domain $\cal U_\lambda(Q)$ as
\begin{equation}\label{eqn:ring_domain_function}
  g_{Q,\lambda}
  = \sum_{E \in \cal U_\lambda(Q)} h_E.
\end{equation}
The ring domain operator onto
$\lin \big\{g_{Q,\lambda}\, :\, Q \in \dcubes \big\}$
is then given by:
\begin{equation}\label{eqn:ring_domain_operator}
  S_\lambda u
  = \sum_{Q \in \dcubes} \langle u, h_Q \rangle\, g_{Q,\lambda} |Q|^{-1}.
\end{equation}

\bigskip

The main tools for analyzing $S_\lambda$ are on one hand Figiel's
shift operators $T_m$, $m \in \bb Z^n$, defined as linear extension of
the map
\begin{equation*}
  T_m h_Q = h_{Q+m\, \sidelength(Q)},
\end{equation*}
and on the other hand Bourgain's version of Stein's martingale
inequality.

\bigskip

Before beginning to analyze our ring domain operator $S_\lambda$,
we decompose $g_{Q,\lambda}$ into a sum of no more than $3 n$
functions, well localised in the vicinity of the set of the
discontinuities of the Haar function $h_Q$.
So for any $Q \in \dcubes$ we partition
\begin{equation*}
  \cal U_\lambda(Q) = \Union_{i = 1}^{3n} \cal U_\lambda^{(i)}(Q),
\end{equation*}
such that for all $E \in \cal U_\lambda^{(i)}(Q)$ holds
\begin{equation*}
  \cal U_\lambda^{(i)}(Q)
  \subset \big\{ E + j\, u_i\, :\, j \in \bb Z \big\},
\end{equation*}
where $u_i$ is one of the standard unit vectors of $\bb R^n$.
This partition induces a splitting of the blocks $g_{Q,\lambda}$ into
\begin{equation*}
  g_{Q,\lambda} = \sum_{i=1}^{3n} g_{Q,\lambda}^{(i)}.
\end{equation*}

We denote one of the functions $g_{Q,\lambda}^{(i)}$ decomposing
$g_{Q,\lambda}$ generically by $g_{Q,\lambda}$ again, and we may
assume that the support of $g_{Q,\lambda}$ is aligned orthogonal to
$e_1 \in \bb R^n$, where $e_1 = (1,0,\ldots,0)$.
We split the operator $S_\lambda$ accordingly, and denote the
operator aligned orthogonal to $e_1$ by $S_\lambda$ again. So we
have analogously to equation~\eqref{eqn:ring_domain_operator}
\begin{equation*}
  S_\lambda u
  = \sum_{Q \in \dcubes} \langle u, h_Q \rangle\, g_{Q,\lambda} |Q|^{-1},
\end{equation*}
with the support of $g_{Q,\lambda}$ now beeing localized in the
vicinity of just one of the $3$ faces of $D(Q)$ perpendicular to
$e_1$.

Recalling \eqref{eqn:shift_operator_1} it is easy to see that for any
$u = \sum_{Q \in \dcubes} u_Q\, h_Q\, |Q|^{-1} \in L_X^p$ one can find
functions $\{c_Q\}_{Q \in \dcubes}$, $|c_Q| = 1$ such that
\begin{equation*}
  \sum_{Q \in \dcubes} c_Q\, u_Q\, h_Q\, |Q|^{-1}
  = \sum_{m=0}^{2^\lambda-1} T_{m\,e_1} S_\lambda u
  = \sum_{m=0}^{2^\lambda-1} S_\lambda^m u,
\end{equation*}
where we defined
\begin{equation}\label{eqn:ring_domain_operator--shifted}
  S_\lambda^m u = T_{m\,e_1} S_\lambda u.
\end{equation}
Precisely, $c_Q$ is given by
\begin{equation*}
  c_Q = h_Q \cdot \sum_{m=0}^{2^\lambda-1} T_{m\, e_1} g_{Q,\lambda}.
\end{equation*}
Note that the shifted support strips $U_\lambda(Q)$ of
$T_{m\,e_1} g_{Q,\lambda}$, $0 \leq m < 2^\lambda$ cover the whole
cube $Q$ (see figure~\vref{pic:shifting_a_strip}).
\begin{figure}[bt]
  \centering
  \includegraphics{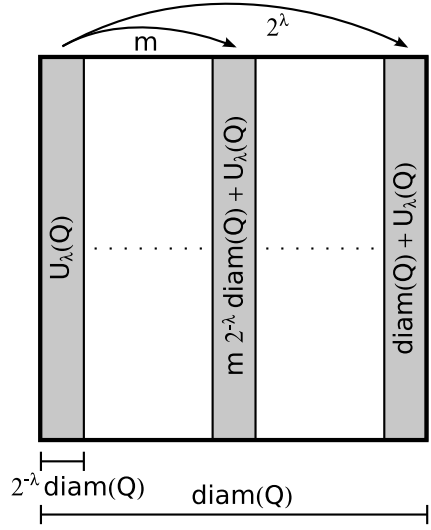}
  \caption{The thin blocks $g_{Q,\lambda}$ with shaded support
    $U_\lambda(Q)$ are shifted to cover the whole cube Q.}
  \label{pic:shifting_a_strip}
\end{figure}
So the well known \umd--property and Kahane's contraction principle
imply
\begin{equation}\label{eqn:ring_domain_operator--sum}
  \|u\|_{L_X^p} \approx
  \big\|
    \sum_{Q \in \dcubes} c_Q\, u_Q\, h_Q\, |Q|^{-1}
  \big\|_{L_X^p}
  =
  \big\|
    \sum_{m=0}^{2^\lambda-1} S_\lambda^m u
  \big\|_{L_X^p}.
\end{equation}

\subsection{Estimates for the Ring Domain Operator}\label{ss:ring_domain_operator-estimates}\hfill

The next Lemma analyzes the spectrum of $S_\lambda$ and prepares for
the construction of atoms, used later in the martingale estimates for
$S_\lambda$.

Before we state the lemma, we build up some notation.
Let $\pred_\lambda\, :\, \dint \rightarrow \dint$, and define for any
$I \in \dint$
\begin{equation*}
  \pred_\lambda(I) = J,
\end{equation*}
where the uniquely determined $J \in \dint$ is such that
$|J| = 2^\lambda\, |I|$ and $J \supset I$.
Furthermore let
\begin{equation*}
  \scr B \subset
  \{ I \in \dint\, :\, \inf I = \inf \pred_\lambda(I) \},
\end{equation*}
such that for all $J,K \in \scr B$ with $|J| \neq |K|$ holds that
\begin{equation*}
  |J| \leq \frac{1}{4}\, |K|
  \quad \text{or} \quad
  |K| \leq \frac{1}{4}\, |J|.
\end{equation*}

\begin{lem}\label{lem:shift-lemma}
  For any $\lambda \geq 1$ let $0 \leq m \leq 2^{\lambda-1}$,
  \begin{equation*}
    \tau(I) = I + m\, |I|, \qquad I \in \dint,
  \end{equation*}
  then
    \begin{equation*}
    \big|
      I \isect
      \Union_{d = 1}^{\lambda-1}
      \Union_{\substack{J \in \scr B\\ |J| = 2^{-d}\, |I|}}
      J \union \tau_m(J)
    \big|
    \leq \frac{2}{3}\, |I|,
  \end{equation*}
  for all $I \in \scr B$.
\end{lem}

\begin{proof}
  First we claim that for any $I \in \scr B \union \tau_m(\scr B)$,
  $1 \leq d \leq \lambda - 1$ and $J,K \in \scr B$ with
  $|J| = |K| = 2^{-d}\, |I|$ holds
  \begin{equation}\label{eqn:lem:filtration:uniqueness}
    ( J \union \tau_m(J) ) \isect I \neq \emptyset
    \quad \text{and} \quad
    ( K \union \tau_m(K) ) \isect I \neq \emptyset
    \quad \text{implies} \quad
    J = K.
  \end{equation}
  If we assume this claim does not hold true, then we can find
  intervals $J \neq K$ such that
  \begin{equation*}
    ( J \union \tau_m(J) ) \isect I \neq \emptyset
    \quad \text{and} \quad
    ( K \union \tau_m(K) ) \isect I \neq \emptyset.
  \end{equation*}
  Since $J \neq K$ we know from the definition of $\scr B$ that
  \begin{equation*}
    \dist (\tau_m(J),\tau_m(K))
    = \dist (J,K)
    \geq (2^\lambda - 1)\, |J|,
  \end{equation*}
  consequently
  \begin{equation*}
    \dist (J \union \tau_m(J),K \union \tau_m(K))
    \geq (2^\lambda - 1 - m)\, |J|.
  \end{equation*}
  Since $I$ intersects both $J \union \tau_m(J)$ and
  $K \union \tau_m(K)$, we infer
  \begin{align*}
    |I|
    & \geq \dist (J \union \tau_m(J),K \union \tau_m(K)) + 2\, |J|\\
    & \geq (2^\lambda - m + 1))\, 2^{-d}\, |I|\\
    & \geq (2^{\lambda-1} + 1)\, 2^{-d}\, |I|\\
    & > |I|,
  \end{align*}
  which is a contradiction.
  Hence~\eqref{eqn:lem:filtration:uniqueness} holds, wich means that
  if $1 \leq d \leq \lambda -1$, any interval
  $I \in \scr B \union \tau_m(\scr B)$ intersects at most one of the
  sets
  \begin{equation*}
    \{ J \union \tau_m(J) \in \scr B\, :\, |J| = 2^{-d}\, |I| \}.
  \end{equation*}
  If such a $J$ exists we denote it by $J_d(I) \in \scr B$, and define
  $J_d(I) = \emptyset$ otherwise.
  Note that for small shift widths $m$ or small $J$ it may happen
  that $J_d(I) \union \tau_m(J_d(I)) \subset I$.

  Using~\eqref{eqn:lem:filtration:uniqueness} we see that for
  every $I \in \scr B \union \scr \tau_m(\scr B)$
  \begin{align*}
    \big|
      I \isect
      \Union_{d = 1}^{\lambda-1}
      \Union_{\substack{J \in \scr B\\ |J| = 2^{-d}\, |I|}}
      J \union \tau_m(J)
    \big|
    & = \sum_{d = 1}^{\lambda-1}
      \big| I \isect \big( J_d(I) \union \tau_m(J_d(I)) \big) \big|\\
    & \leq \sum_{d = 1}^{\lambda-1} 2 \cdot |J_d(I)|\\
    & \leq 2 \cdot \sum_{d = 1}^\infty 2^{-2d}\, |I|\\
    & = \frac{2}{3}\, |I|.
  \end{align*}
  The last inequality holds since for any $J,K \in \scr B$,
  $|J| \neq |K|$ implies $|J| \leq \frac{1}{4} |K|$ or
  $|K| \leq \frac{1}{4} |J|$, thus finishing the proof of the lemma.
\end{proof}

Having verified lemma~\vref{lem:shift-lemma}, we now turn to prove the
following \emph{pointwise} estimates for $S_\lambda$.
There exists a constant $C>0$ such that
\begin{equation*}
  \frac{1}{C}\, \|S_\lambda^k u\|_{L_X^p}
  \leq \|S_\lambda^m u\|_{L_X^p}
  \leq C\, \|S_\lambda^k u\|_{L_X^p},
  \qquad 
  \lambda \geq 0,\ 0 \leq k,m < 2^\lambda,\ u \in L_X^p.
\end{equation*}

Once more we emphazise that these estimates are crucial for the proof
of the main result~\eqref{eqn:main_result}.

\begin{pro}\label{pro:ring_operator_equivalence}
  Let $X$ be a $\umd$ space, $1 < p < \infty$ and $n \in \bb N$.
  There exists $0 < C < \infty$ such that for any
  $k, m, \lambda \in \bb Z$, with $\lambda \geq 0$,
  $0 \leq k \leq 2^\lambda -1$ and $0 \leq m \leq 2^\lambda -1$ the
  estimate
  \begin{equation}\label{eqn:ring_operator_equivalence}
    \frac{1}{C}\, \|S_\lambda^k u\|_{L_X^p}
    \leq \|S_\lambda^m u\|_{L_X^p}
    \leq C\, \|S_\lambda^k u\|_{L_X^p}
  \end{equation}
  holds true for all $u \in L_X^p$.
  The constant $C$ depends on $n$, $p$ and $X$, particularly on the
  constant arising in Stein's martingale inequality.
\end{pro}

\begin{proof}
  In order to show~\eqref{eqn:ring_operator_equivalence}, we will
  first prove
  \begin{equation}\label{eqn:ring_operator_equivalence--reduced-1}
    \frac{1}{C}\, \|S_\lambda^0 u\|_{L_X^p}
    \leq \|S_\lambda^m u\|_{L_X^p}
    \leq C\, \|S_\lambda^0 u\|_{L_X^p},
  \end{equation}
  for all $\lambda \geq 0$ and $0 \leq m \leq
  2^{\lambda-1}$. Exploiting symmetry will also establish
  \begin{equation}\label{eqn:ring_operator_equivalence--reduced-2}
    \frac{1}{C}\, \|S_\lambda^{2^\lambda -1} u\|_{L_X^p}
    \leq \|S_\lambda^m u\|_{L_X^p}
    \leq C\, \|S_\lambda^{2^\lambda -1} u\|_{L_X^p},
  \end{equation}
  for all $\lambda \geq 0$ and
  $2^{\lambda-1} -1 \leq m \leq 2^\lambda -1$.
  Once we have~\eqref{eqn:ring_operator_equivalence--reduced-1} and
  \eqref{eqn:ring_operator_equivalence--reduced-2}, we
  gain~\eqref{eqn:ring_operator_equivalence}, since all operators
  $\{S_\lambda^m\}_m$ are uniformly equivalent to
  $S_\lambda^{2^{\lambda-1}}$ (and to $S_\lambda^{(2^{\lambda-1}-1)}$).

  \smallskip

  We begin the proof defining
  \begin{equation*}
    \scr C = \{
      Q \in \dcubes\, :\,
      \inf_{q \in Q} \langle q, e_1 \rangle
      = \inf_{p \in \pi_\lambda(Q)} \langle p, e_1 \rangle
    \}
  \end{equation*}
  and the four collections
  \begin{align*}
    \scr B_{\text{odd}}^{0} & = \Union_{j \in \bb Z}
      \Union_{\substack{k = 0\\k \text{ odd}}}^{\lambda-1}
      \scr C \isect \dint_{2 j \lambda + k},
    & \scr B_{\text{even}}^{0} & = \Union_{j \in \bb Z}
      \Union_{\substack{k = 0\\k \text{ even}}}^{\lambda-1}
      \scr C \isect \dint_{2 j \lambda + k},\\
    \scr B_{\text{odd}}^{1} & = \Union_{j \in \bb Z}
      \Union_{\substack{k = 0\\k \text{ odd}}}^{\lambda-1}
      \scr C \isect \dint_{(2j + 1) \lambda + k},
    & \scr B_{\text{even}}^{1} & = \Union_{j \in \bb Z}
      \Union_{\substack{k = 0\\k \text{ even}}}^{\lambda-1}
      \scr C \isect \dint_{(2j + 1) \lambda + k},
  \end{align*}
  each generically denoted by $\scr B$.
  Apparently $\scr C$ is exactely the Haar--spectrum of $S_\lambda^0$,
  and the collection $\scr B$ was constructed such that we may apply
  lemma~\vref{lem:shift-lemma} to its blocks of $\lambda-1$
  consecutive levels with every second scale stripped off.

  With $\scr B$ fixed we claim the existence of a filtration
  $\{\cal F_j\}_j$ such that for every $j \in \bb Z$ and
  $Q \in \scr B \isect \dcubes_j$ exists an atom $A(Q)$ of $\cal F_j$
  satisfying the inequalities
  \begin{align}\label{eqn:atom-properties}
    |A(Q)| & \leq 2\, |Q|,
    & |Q \isect A(Q)| & \geq \frac{1}{3}\, |Q|,
    & |\tau_m(Q) \isect A(Q)| & \geq \frac{1}{3}\, |Q|.
  \end{align}
  We shall use an auxiliary argument regarding overlaps of dyadic
  cubes from different $\lambda-1$--blocks, exploiting that cubes from
  different $\lambda-1$--blocks are seperated by at least $\lambda$
  levels.
  This will become more obvious when considering the following
  argument.
  Let $\tau_m$ be the right--shift operation in direction $e_1$,
  precisely
  \begin{equation*}
    \tau(Q) = Q + 2m\, \sidelength(Q)\,  e_1,
  \end{equation*}
  for all $Q \in \dcubes$.

  Now for each $Q \in \scr B$ we will define atoms inductively,
  beginning at the finest level of a $\lambda-1$ block.
  More precisely, fix an arbitrary $b \in \bb Z$ such that for
  any $Q, Q' \in \scr B$ with $|Q| = 2^{-b\, n}$ and $|Q'| < |Q|$
  follows $|Q'| \leq 2^{-\lambda\, n}\, |Q|$.
  Initially define 
  \begin{equation}\label{eqn:dfn:atom-1}
    A(Q) = Q \union \tau_m(Q),
  \end{equation}
  for $Q \in \scr B \isect \dcubes_b$.
  Assume we already constructed atoms on the scales $b,b-1,\ldots,j$,
  while $j -1 \geq b - (\lambda - 1)$, then define for all
  $Q \in \scr B \isect \dcubes_{j-1}$
  \begin{equation}\label{eqn:dfn:atom-2}
    A(Q) = \big( Q \union \tau_m(Q) \big) \setminus
    \big(
      \Union_{k = j}^b \Union_{M \in \scr B \isect \dcubes_k} A(M)
    \big).
  \end{equation}
  Applying lemma~\vref{lem:shift-lemma} in direction $e_1$ to the
  atoms $A(Q) \subset Q \union \tau_m(Q)$ inside the block
  $b, b-1, \ldots ,b - (\lambda-1)$ we gain
  \begin{equation*}
    |Q \isect A(Q)|
    = |Q| - |
      Q \isect \Union_{k = j}^b \Union_{M \in \scr B \isect \dcubes_k}
      A(M)
    |
    \geq \frac{1}{3}\, |Q|,
  \end{equation*}
  and analogously
  \begin{equation*}
    |\tau_m(Q) \isect A(Q)| \geq \frac{1}{3}\, |Q|,
  \end{equation*}
  which yields~\eqref{eqn:atom-properties}.
  Finally we define the collection
  \begin{equation}\label{eqn:dfn:atoms_of_interest}
    \cal A_j =
    \big\{
      A(Q)\, :\, Q \in \scr B \isect \dcubes_j
    \big\},
  \end{equation}
  and the filtration
  \begin{equation}\label{eqn:dfn:filtration}
    \cal F_j = \salg\Big( \Union_{i \leq j} \cal A_i \Big).
  \end{equation}
  What is left to show is that every $A \in \cal A_j$ is an atom for
  the $\sigma$--algebra $\cal F_j$.

  To see this we argue as follows.
  First note that any two atoms are either localized in the same
  $\lambda-1$--block, or are seperated by at least $\lambda$ levels.
  If atoms $A(Q)$ and $A(Q')$ are in the same $\lambda-1$--block,
  then they do not intersect per construction.
  If $A(Q)$ and $A(Q')$ intersect and
  $|Q'| \leq 2^{-\lambda\, n}\, |Q|$, then since
  \begin{equation*}
    A(Q') \subset (Q' \union \tau_m(Q')) \subset \pi_\lambda(Q')
  \end{equation*}
  we have
  \begin{equation*}
    \pi_\lambda(Q') \isect A(Q) \neq \emptyset. 
  \end{equation*}
  Clearly, $A(Q)$ comprises of cubes $K$ which are at least as big as
  $\pi_\lambda(Q')$, so $|\pi_\lambda(Q')| \leq |K|$ and consequently
  \begin{equation*}
    A(Q') \subset A(Q).
  \end{equation*}
  This means that $\Union_j \cal A_j$ is a nested collections of sets,
  hence every $A \in \cal A_j$ is an atom for the $\sigma$--algebra
  $\cal F_j$.
    
  \bigskip

  Now, after all this preparation we are about to finish the proof.
  Having~\eqref{eqn:atom-properties} at hand and knowing that
  the collection $\cal A_j$ are atoms for $\cal F_j$ one can find a
  constant $C$ depending only on the constants arising
  in~\eqref{eqn:atom-properties} such that
  \begin{equation}\label{eqn:atom-propagation-1}
    \frac{1}{C}\, \cond \big(
      (S_\lambda^0 u)_j\, |\, \cal F_j
    \big)
    \leq \cond \big(
      (S_\lambda^m u)_j\, |\, \cal F_j
    \big)
    \leq C\, \cond \big(
      (S_\lambda^0 u)_j\, |\, \cal F_j
    \big)
  \end{equation}
  where $(S_\lambda^0 u)_j$ and $(S_\lambda^m u)_j$ denote the
  restriction of the Haar expansion of $S_\lambda^0 u$ and
  $(S_\lambda^m u)_j$ to dyadic cubes in $\dcubes_j$, respectively.
  Furthermore one can see that
  \begin{align}
    (S_\lambda^0 u)_j
    & \leq C\, \cond \big(
      \cond \big( (S_\lambda^0 u)_j\, |\, \cal F_j \big)\,
      \big|\, \dcubes_j
    \big),
    \label{eqn:atom-propagation-2}
    \intertext{and similarly}
    (S_\lambda^m u)_j
    & \leq C\, \cond \big(
      \cond \big( (S_\lambda^m u)_j\, |\, \cal F_j \big)\,
      \big|\, \dcubes_j
    \big).
    \label{eqn:atom-propagation-3}
  \end{align}

  Initially, by the $\umd$--property
  \begin{equation*}
    \|S_\lambda^0 u\|_{L_X^p}^p
    \approx \int_0^1 \big\|
      \sum_{j \in \bb Z} r_j(t)\, (S_\lambda^0 u)_j
    \big\|_{L_X^p}^p
    \mathrm dt,
  \end{equation*}
  which together with Kahane's contraction principle applied
  to~\eqref{eqn:atom-propagation-2} yields
  \begin{equation*}
    \|S_\lambda^0 u\|_{L_X^p}^p
    \lesssim \int_0^1 \big\|
      \sum_{j \in \bb Z} r_j(t)\, \cond \big(
      \cond \big( (S_\lambda^0 u)_j\, |\, \cal F_j \big)\,
      \big|\, \dcubes_j
    \big)
    \big\|_{L_X^p}^p.
  \end{equation*}
  Issuing Stein's martingale
  inequality~\eqref{eqn:steins_martinagle_inequality} for the
  filtration $\{\dcubes_j\}_j$
  gives
  \begin{equation*}
    \|S_\lambda^0 u\|_{L_X^p}^p
    \lesssim \int_0^1 \big\|
      \sum_{j \in \bb Z} r_j(t)\,
      \cond \big( (S_\lambda^0 u)_j\, |\, \cal F_j \big)
    \big\|_{L_X^p}^p,
  \end{equation*}
  which is in view of~\eqref{eqn:atom-propagation-1} and Kahane's
  contraction principle dominated by a constant multiple of
  \begin{equation*}
    \big\|
      \sum_{j \in \bb Z} r_j(t)\,
      \cond \big( (S_\lambda^m u)_j\, |\, \cal F_j \big)
    \big\|_{L_X^p}^p.
  \end{equation*}
  This time we apply Stein's martingale inequality to the filtration
  $\{\cal F_j\}_j$, and subsequently make use of the $\umd$--property
  to dispose of the Rademacher functions, hence
  \begin{equation*}
    \|S_\lambda^0 u\|_{L_X^p}^p
    \lesssim \int_0^1 \big\|
      \sum_{j \in \bb Z} r_j(t)\, (S_\lambda^m u)_j
    \big\|_{L_X^p}^p
    \approx \|S_\lambda^m u\|_{L_X^p}^p.
  \end{equation*}

  Repeating this argument with $S_\lambda^0$ and $S_\lambda^m$
  interchanged and using~\eqref{eqn:atom-propagation-3} instead
  of~\eqref{eqn:atom-propagation-2} we get the converse inequality
  \begin{equation*}
    \|S_\lambda^m u\|_{L_X^p}^p
    \lesssim \| S_\lambda^0 u \|_{L_X^p}^p,
  \end{equation*}
  a fortiori we obtain
  \eqref{eqn:ring_operator_equivalence--reduced-1}, that was
  \begin{equation*}
    \frac{1}{C}\, \|S_\lambda^0 u\|_{L_X^p}
    \leq \| S_\lambda^m u \|_{L_X^p}
    \leq C\, \|S_\lambda^0 u\|_{L_X^p},
  \end{equation*}
  for all $\lambda \geq 0$, $0 \leq m \leq 2^{\lambda -1}$ and
  $u \in L_X^p$,
  where $C$ depends only on $n$, $p$ and $X$.

  \bigskip

  Observe that due to symmetry we may use the same argument for the
  operators $S_\lambda^m$, $2^{\lambda -1} \leq m \leq 2^\lambda -1$,
  when we reverse the sign of the shift operation and replace
  $S_\lambda^0$ by $S_\lambda^{2^\lambda -1}$.
  Therefore
  inequality~\eqref{eqn:ring_operator_equivalence--reduced-2} holds
  true
  \begin{equation*}
    \frac{1}{C}\, \|S_\lambda^{2^\lambda -1} u\|_{L_X^p}
    \leq \| S_\lambda^m u \|_{L_X^p}
    \leq C\, \|S_\lambda^{2^\lambda -1} u\|_{L_X^p},
  \end{equation*}
  for all $\lambda \geq 0$,
  $2^{\lambda -1} -1 \leq m \leq 2^\lambda -1$ and $u \in L_X^p$,
  where $C$ depends only on $n$, $p$ and $X$.

  Joining the last two inequalities via $S_\lambda^{2^{\lambda-1}}$ (or
  $S_\lambda^{2^{\lambda-1}-1}$) concludes the proof of the
  proposition.
\end{proof}

\begin{rem}
  By symmetry it is easy to see that
  \begin{equation*}
    \|S_\lambda^m\, :\, L_X^p \rightarrow L_X^p \|
    \approx \|S_\lambda^{2^\lambda -1 -m}
      \, :\, L_X^p \rightarrow L_X^p \|
  \end{equation*}
  holds true for all $0 \leq m \leq 2^\lambda -1$.
  Unfortunately, this does not help us with our pointwise estimates.
\end{rem}

However, we are now about to prove the main result on ring domain
operators
\begin{thm}\label{thm:ring_domain_operator}
  For $\lambda \geq 0$ let $S_\lambda$ denote the ring domain operator
  defined by
  \begin{equation*}
    S_\lambda u
    = \sum_{Q \in \dcubes}
      \langle u, h_Q \rangle\, g_{Q,\lambda} |Q|^{-1}.
  \end{equation*}
  When $L_X^p$ has cotype $\cal C(L_X^p)$, there exists a constant
  $C > 0$ such that for every $u \in L_X^p$ and $\lambda \geq 0$
  \begin{equation}\label{eqn:ring_domain_operator_estimate}
    \| S_\lambda u \|_{L_X^p}
    \leq C\, 2^{-\lambda/\cal C(L_X^p)}\, \|u\|_{L_X^p},
  \end{equation}
  where the constant $C$ depends only on $n$, $p$, $X$ and
  $\cal C(L_X^p)$.
\end{thm}

\begin{proof}
  A simple application of Kahane's contraction principle shows that
  the estimate holds if we restrict $\lambda$ to
  $0 \leq \lambda \leq 1$.

  \smallskip

  So from now on we may assume $\lambda \geq 2$.

  For $0 \leq m \leq 2^\lambda -1$ we defined
  in~\eqref{eqn:ring_domain_operator--sum}, the shifted ring domain
  operators
  \begin{equation*}
    S_\lambda^m u = T_{m\,e_1} S_\lambda u.
  \end{equation*}
  Observe that for all $u \in L_X^p$ and
  $0 \leq k \neq m \leq 2^\lambda -1$
  \begin{equation}\label{ring_domain_operator--shifted--disjoint}
    \big\{ Q \in \dcubes\, :\,
      \langle S_\lambda^k u, h_Q \rangle \neq 0
    \big\}
    \isect
    \big\{ Q \in \dcubes\, :\,
      \langle S_\lambda^m u, h_Q \rangle \neq 0
    \big\}
    = \emptyset.
  \end{equation}

  According to~\eqref{eqn:ring_domain_operator--sum} we know
  \begin{equation*}
    \|u\|_{L_X^p}
    \approx \big\|
      \sum_{m=0}^{2^\lambda-1} S_\lambda^m u
    \big\|_{L_X^p},
  \end{equation*}
  and since~\eqref{ring_domain_operator--shifted--disjoint} enables
  us to use the cotype inequality in $L_X^p$ we gain
  \begin{equation*}
    \|u\|_{L_X^p}
    \gtrsim \bigg(
      \sum_{m=0}^{2^\lambda-1}
      \big\| S_\lambda^m u \big\|_{L_X^p}^{\cal C(L_X^p)}
    \bigg)^{1/\cal C(L_X^p)}.
  \end{equation*}
  Applying the result~\eqref{eqn:ring_operator_equivalence} of
  proposition~\vref{pro:ring_operator_equivalence}, which guarantees
  that the operators $S_\lambda^m$, $0 \leq m \leq 2^\lambda -1$ are
  \emph{pointwise} norm--equivalent to $S_\lambda = S_\lambda^0$,
  reveals the end of the proof
  \begin{equation*}
    \|u\|_{L_X^p}
    \gtrsim \bigg(
      \sum_{m=0}^{2^\lambda-1}
      \big\| S_\lambda u \big\|_{L_X^p}^{\cal C(L_X^p)}
    \bigg)^{1/\cal C(L_X^p)}
    = 2^{\lambda/\cal C(L_X^p)}\, \big\| S_\lambda u \big\|_{L_X^p}.
  \end{equation*}
\end{proof}

Repeating the proof of theorem~\ref{thm:ring_domain_operator}
without proposition~\ref{pro:ring_operator_equivalence} using
Figiel's bound on shift
operators~\eqref{eqn:shift_operator_1_estimate} directly, would result
in
\begin{equation*}
  \| S_\lambda u \|_{L_X^p}
  \leq C\, \lambda^{\alpha}\,  2^{-\lambda/\cal C(L_X^p)}\,
    \|u\|_{L_X^p},
\end{equation*}
where $L_X^p$ has cotype $\cal C(L_X^p)$, and the constant $C$ depends
only on $n$, $p$, $X$ and $\cal C(L_X^p)$. The exponent
$0 < \alpha < 1$ is the exponent occurring in Figiel's
estimate~\eqref{eqn:shift_operator_1_estimate}.

\newpage
\section{Estimates for $P_l^{(\varepsilon)}$ and $P_l^{(\varepsilon)} R_{i_0}^{-1}$}\label{s:estimates_mollified}\hfill

For any $u \in L_X^p$ with $1 < p < \infty$ fixed, define
$ P^{(\varepsilon)}\, :\, L_X^p \longrightarrow L_X^p$ by setting
\begin{equation}\label{eqn:directional_projection}
  P^{(\varepsilon)} u
  = \sum_{Q \in \dcubes}
    \langle u, h_Q^{(\varepsilon)} \rangle\,
    h_Q^{(\varepsilon)}\, |Q|^{-1}.
\end{equation}

In order to estimate the directional Haar projection operator
$P^{(\varepsilon)}$, we will decompose $P^{(\varepsilon)}$ in
subsection~\vref{ss:decomposition} into a series of mollified
operators $\sum_l P_l^{(\varepsilon)}$,
following~\cite{lee_mueller_mueller:2007}.
Subsequentely, J. Lee, P. F. X. Mueller and S. Mueller used wavelet
expansions to further analyze $P_l^{(\varepsilon)}$.

However, in this paper $P_l^{(\varepsilon)}$ is decomposed into a
series of ring domain operators
$\sum_{\lambda(l)} c_{\lambda(l)}S_{\lambda(l)}$, using martingale
methods complying with $\umd$--spaces.

This is done in section~\vref{s:estimates_mollified}, where all
non--trivial estimates for the operators $P_l^{(\varepsilon)}$ and
$P_l^{(\varepsilon)} R_{i_0}^{-1}$ are obtained from the inequalities
for the ring domain operators $S_\lambda$, analyzed in
section~\ref{s:ring_domain_operator}.

\subsection{Decomposition of $P^{(\varepsilon)}$}\label{ss:decomposition}\hfill

We give a brief overview of the Littlewood--Paley decomposition used
in~\cite{lee_mueller_mueller:2007}, and continue with further
decompositions in subsection~\ref{ss:integral_kernel}
and~\ref{ss:decomposition_estimates} suited for the $\umd$--domain.

\bigskip

As in~\cite{lee_mueller_mueller:2007}, we employ a compactly
supported, smooth approximation of the identity, to obtain a
decomposition of the directional projection $P^{(\varepsilon)}$ into a
series of mollified operators
\begin{equation}\label{eqn:directional_projection--decomposition}
  P^{(\varepsilon)} = \sum_{l \in \bb Z} P_l^{(\varepsilon)}.
\end{equation}

First we fix $b \in C_c^\infty(]0,1[^n)$ such that
\begin{equation}\label{eqn:decomposition:vanishing_moments}
  \gint{x}{b(x)} = 1,
  \quad \text{and} \quad
  \gint{x_i}{x_i\, b(x_1,\ldots,x_i,\ldots,x_n)} = 0,
\end{equation}
for all $1 \leq i \leq n$.
This can be easily achieved in the Fourier domain.
Let $l \in \bb Z$ and define
\begin{equation}\label{eqn:decomposition:convolution}
  \Delta_l u = u * d_l,
  \quad \text{where} \quad
  d_l(x) = 2^{ln} d(2^l x)
  \quad \text{and} \quad
  d(x) = 2^n\, b(2\, x) - b(x).
\end{equation}
For any $u \in L_X^p(\bb R^n)$ holds that
\begin{equation}\label{eqn:decomposition:convolution-1}
  u = \sum_{l \in \bb Z} \Delta_l u,
\end{equation}
where the series converges in $L_X^p$.
Denoting $\dcubes_j \subset \dcubes$ the collection of all dyadic
cubes having measure $2^{-jn}$, we set
\begin{equation}\label{eqn:dfn:mollified_operator--1}
  P_l^{(\varepsilon)} u
  = \sum_{j \in \bb Z} \sum_{Q \in \dcubes_j}
    \langle u, \Delta_{j+l}(h_Q^{(\varepsilon)}) \rangle\,
    h_Q^{(\varepsilon)}\, |Q|^{-1}
\end{equation}
and observe that by~\eqref{eqn:decomposition:convolution-1} for all
$u \in L_X^p$
\begin{equation*}
  P^{(\varepsilon)} u = \sum_{l \in \bb Z} P_l^{(\varepsilon)} u,
\end{equation*}
where equality holds in the sense of $L_X^p$.
Setting $f_{Q,l}^{(\varepsilon)} = \Delta_{j+l} h_Q^{(\varepsilon)}$, if
$Q \in \dcubes_j$, we rewrite~\eqref{eqn:dfn:mollified_operator--1}
as 
\begin{equation}\label{eqn:dfn:mollified_operator--2}
  P_l^{(\varepsilon)} u
  = \sum_{Q \in \dcubes}
    \langle u, f_{Q,l}^{(\varepsilon)} \rangle\,
    h_Q^{(\varepsilon)}\, |Q|^{-1}.
\end{equation}

\subsection{The Integral Kernel of $P_l^{(\varepsilon)}$}\label{ss:integral_kernel}\hfill

In this subsection we identify the integral kernel
$K_l^{(\varepsilon)}$ of the operator $P_l^{(\varepsilon)}$ and expand
it in its Haar series, exploiting Figiel's martingale approach.
At this point we deviate significantly from the methods
of~\cite{lee_mueller_mueller:2007}.

We intend to use martingale methods on the operators
$P_l^{(\varepsilon)}$, therefore will take a close look at their
kernels $K_l^{(\varepsilon)}$,
\begin{align}
  \big(P_l^{(\varepsilon)}u\big) (x)
  = \gint{y}{K_l^{(\varepsilon)}(x,y)\, u(y)},
  \label{eqn:mollified_operator}
  \intertext{where}
  K_l^{(\varepsilon)}(x,y)
  = \sum_{Q \in \dcubes}
    h_Q^{(\varepsilon)}(x)\, f_{Q,l}^{(\varepsilon)}(y)\, |Q|^{-1}.
  \label{eqn:mollified_kernel}
\end{align}

Expanding $K_l^{(\varepsilon)}$ according to Figiel's approach into
the series
\begin{equation}\label{eqn:kernel--figiel_expansion}
  \sum_{
    \substack{
      \alpha, \beta \in \{0,1\}^n\\
      (\alpha,\beta) \neq 0
    }
  }
  \sum_{
    \substack{
      K,M,Q \in \dcubes\, :\\
      |K| = |M|
    }
  }
  \langle h_Q^{(\varepsilon)}, h_K^{(\alpha)} \rangle
  \langle f_{Q,l}^{(\varepsilon)}, h_M^{(\beta)} \rangle
  |K|^{-1} |M|^{-1} |Q|^{-1}
  h_K^{(\alpha)}(x)\, h_M^{(\beta)}(y),
\end{equation}
we will have to distinguish the following settings for the parameter
$\beta$:
\begin{enumerate}
\item $\beta \neq 0$,\label{enu:beta_nonzero}
\item $\beta = 0$\label{enu:beta_zero}.
\end{enumerate}
Note that due to the condition $(\alpha, \beta) \neq 0$,
case~\eqref{enu:beta_zero} certainly implies $\alpha \neq 0$.

To ease the notation, we will make use of the following convention.
We shall write $h_Q$, denoting one of the functions $h_Q^{(\gamma)}$,
$\gamma \in \{0,1\}^n\setminus \{0\}$, and $\charfun_Q$ for the
characteristic funtion $h_Q^0$.
We may do so since the \umd--property and Kahane's contraction
principle enable us to interchange equally supported Haar functions
having zero mean.

Using this notation, then the Figiel
expansion~\eqref{eqn:kernel--figiel_expansion} according to the
two different cases $(\beta \neq 0)$ and $(\beta = 0$,
$\alpha \neq 0)$ \emph{both} read
\begin{equation}\label{eqn:kernel--figiel_expansion_simplified}
  K_l(x,y) = 
  \sum_{M,Q \in \dcubes}
  \langle f_{Q,l}, h_M \rangle
  |M|^{-1} |Q|^{-1} h_Q(x)\, h_M(y).
\end{equation}
This is exactly the Haar expansion of $K_l$ in the $y$--coordinate,
actually not so surprising since we initially
had Haar functions in the $x$--coordinate
(see \eqref{eqn:kernel--figiel_expansion}).
Figiel's expansion in $\bb R^{2n}$ breaks up the Haar functions
$h_Q^{(\varepsilon)}$ into smaller pieces and reassembles them,
subsequently.
We might have seen the algebraic form
\eqref{eqn:kernel--figiel_expansion_simplified} simply by plugging the
Haar series of $u$ into the operator $P_l^{(\varepsilon)}$.
However, after a few purely algebraic manipulations, Figiel's
expansion in both coordinates yields
identity~\eqref{eqn:kernel--figiel_expansion_simplified}.

\smallskip

Now we present an accurate justification for
identity~\eqref{eqn:kernel--figiel_expansion_simplified}.
Therefore, we fix $\beta \in \{0,1\}^n\setminus\{0\}$,
$\alpha \in \{0,1\}^n$ and rewrite
\eqref{eqn:kernel--figiel_expansion}
\begin{align*}
  K_l(x,y)
  & = \sum_{
    \substack{
      K,M,Q \in \dcubes:\\
      |K| = |M|
    }
  }
  \langle h_Q, h_K^{(\alpha)} \rangle\, 
  \langle f_{Q,l}, h_M \rangle\, 
  |K|^{-1} |M|^{-1} |Q|^{-1} h_K^{(\alpha)}(x)\, h_M(y)\\
  & = \sum_{M,Q \in \dcubes}
  \langle f_{Q,l}, h_M \rangle\, |M|^{-1} |Q|^{-1} h_M(y)
  \sum_{
    \substack{
      K \in \dcubes:\\
      |K| = |M|
    }
  }
  \langle h_Q, h_K^{(\alpha)} \rangle
  |K|^{-1} h_K^{(\alpha)}(x).
\end{align*}
In both cases $\alpha = 0$ and $\alpha \neq 0$ the inner sum
\begin{equation*}
  \sum_{
    \substack{
      K \in \dcubes:\\
      |K| = |M|
    }
  }
  \langle h_Q, h_K^{(\alpha)} \rangle\, 
  |K|^{-1} h_K^{(\alpha)}(x)
\end{equation*}
is identically
\begin{equation*}
  h_Q(x) \qquad \text{for all $Q$ and $M$},
\end{equation*}
for beeing either the conditional expectation of
$h_Q$, or exploiting the orthogonality of the Haar basis,
respectively.
Hence we obtain \eqref{eqn:kernel--figiel_expansion_simplified}.

Now let $\beta = 0$, which implies $\alpha \neq 0$ as noted before,
therefore Figiel's expansion~\eqref{eqn:kernel--figiel_expansion}
reads
\begin{align*}
  K_l(x,y)
  & = \sum_{
    \substack{
      K,M,Q \in \dcubes\, :\\
      |K| = |M|
    }
  }
  \langle h_Q, h_K \rangle\,
  \langle f_{Q,l}, \charfun_M \rangle\,
  |K|^{-1} |M|^{-1} |Q|^{-1}\, h_K(x)\, \charfun_M(y)\\
  & = \sum_{
    \substack{
      M,Q \in \dcubes\, :\\
      |M| = |Q|
    }
  }
  \langle f_{Q,l}, \charfun_M \rangle\,
  |M|^{-1} |Q|^{-1} h_Q(x)\, \charfun_M(y).
\end{align*}
Evaluating this expansion on the Haar series of $u$ would correspond
to developing the $y$--component of $K_l(x,y)$ in a mean zero Haar
series, so we proceed
\begin{align*}
  K_l(x,y)
  & = \sum_{
    \substack{
      K,M,Q \in \dcubes\, :\\
      |M| = |Q|
    }
  }
  \langle f_{Q,l}, \charfun_M \rangle\,
  \langle h_K, \charfun_M \rangle\,
  |K|^{-1} |M|^{-1} |Q|^{-1} h_Q(x)\, h_K(y)\\
  & = \sum_{K,Q \in \dcubes}
  h_Q(x)\, h_K(y)\, |K|^{-1} |Q|^{-1} 
  \sum_{
    \substack{
      M \subsetneqq K\\
      |M| = |Q|
    }
  }
  \langle f_{Q,l}, \charfun_M \rangle\,
  \langle h_K, \charfun_M \rangle\,
  |M|^{-1}\\
  & = \sum_{K,Q \in \dcubes}
  h_Q(x)\, h_K(y)\, |K|^{-1} |Q|^{-1} 
  \Big\langle f_{Q,l},
  \sum_{
    \substack{
      M \subsetneqq K\\
      |M| = |Q|
    }
  }
  \charfun_M \, \langle h_K, \charfun_M \rangle\, |M|^{-1}
  \Big\rangle.
\end{align*}
Observe, the inner sum with $K$ and $Q$ fixed is the conditional
expectation of $h_K$ at a finer scale, hence reproducing $h_K$
\begin{equation*}
  \sum_{
    \substack{
      M \subsetneqq K\\
      |M| = |Q|
    }
  }
  \charfun_M \, \langle h_K, \charfun_M \rangle\, |M|^{-1}
  = h_K,
\end{equation*}
and we gain
\begin{equation*}
  K_l(x,y)
  = \sum_{K,Q \in \dcubes}
  \langle f_{Q,l}, h_K \rangle
  |K|^{-1} |Q|^{-1} h_Q(x)\, h_K(y).
\end{equation*}
Note that we may lift the restriction $|Q| < |K|$, since the
sum~\eqref{eqn:kernel--figiel_expansion_simplified} is parametrized
according to the ratio of the diameters of $Q$ and $M$ in
subsection~\ref{ss:decomposition_estimates}, and split using the
triangle inequality.

\bigskip

As a consequence we may assume the generic
expansion~\eqref{eqn:kernel--figiel_expansion_simplified} of the
integral kernel $K_l(x,y)$ in order to estimate $P_l$.

\subsection{Estimates for $P_l^{(\varepsilon)}$}\label{ss:decomposition_estimates}\hfill

After analyzing some basic properties of the mollified Haar functions
$f_{Q,l}^{(\varepsilon)}$ we turn to estimating $P_l^{(\varepsilon)}$,
guided by the behaviour of $f_{Q,l}^{(\varepsilon)}$, which is mostly
rooted in the different shape of the support of the functions
$f_{Q,l}^{(\varepsilon)}$, $l \geq 0$ and
$f_{Q,l}^{(\varepsilon)}$, $l \leq 0$, respectively (compare the
support inclusions in~\eqref{eqn:mollified_haarfun--1}
and~\eqref{eqn:mollified_haarfun--2}).

\bigskip

As indicated before we will dominate each operator
$P_l^{(\varepsilon)}$ by a series of ring domain operators
\begin{equation*}
  P_l^{(\varepsilon)}
  = \sum_{\lambda(l)} c_{\lambda(l)}\, S_{\lambda(l)}.
\end{equation*}
We shall make use of the estimates for the ring domain operators
$S_\lambda$ developed in section~\ref{s:ring_domain_operator}.

\bigskip

Before analyzing the operators $P_l^{(\varepsilon)}$, we want to
find inequalities for the mollified Haar functions
$f_{Q,l}^{(\varepsilon)}$.
Let $D^{(\varepsilon)}(Q)$ denote the set of discontinuities of the
Haar function $h_Q^{(\varepsilon)}$, then
\begin{equation*}
  D_l^{(\varepsilon)}(Q)
  = \{ x \in \bb R^n\, :\,
    \dist(x, D^{(\varepsilon)}(Q)) \leq C\, 2^{-l}\, \diam(Q) \}.
\end{equation*}
If $l \geq 0$, then
\begin{equation}
  \begin{aligned}
    \gint{x}{f_{Q,l}^{(\varepsilon)}(x)} & = 0,
    & \supp f_{Q,l}^{(\varepsilon)} & \subset D_l^{(\varepsilon)}(Q),\\
    |f_{Q,l}^{(\varepsilon)}| & \leq C,
    & \lip(f_{Q,l}^{(\varepsilon)}) & \leq C\, 2^l\, (\diam(Q))^{-1},
  \end{aligned}
  \label{eqn:mollified_haarfun--1}
\end{equation}
and if $l \leq 0$, we have
\begin{equation}
  \begin{aligned}
    \gint{x}{f_{Q,l}^{(\varepsilon)}(x)}
      & = 0,
    & \supp f_{Q,l}^{(\varepsilon)}
      & \subset C\, 2^{|l|} Q,\\
    |f_{Q,l}^{(\varepsilon)}|
      & \leq C\, 2^{-|l|(n+1)},
    & \lip(f_{Q,l}^{(\varepsilon)})
      & \leq C\, 2^{-|l|(n+2)}\, (\diam(Q))^{-1}.
  \end{aligned}
  \label{eqn:mollified_haarfun--2}
\end{equation}

The different behaviour of the functions $f_{Q,l}^{(\varepsilon)}$
appearing in the definition of the Operators $P_l^{(\varepsilon)}$ for
different signs of $l$ induces the cases $l \geq 0$ and $l \leq 0$.

\subsubsection{Estimates for  $P_l^{(\varepsilon)}$, $l \geq 0$}\label{sss:decomposition_estimates--l>=0}\hfill

At first the operator $P_l$ will be splitted according to
inequalities~\eqref{eqn:mollified_haarfun_coeff--l>=0_case_1},
\eqref{eqn:mollified_haarfun_coeff--l>=0_case_2} and
\eqref{eqn:mollified_haarfun_coeff--l>=0_case_3} into
\begin{equation*}
  P_l = A_l + B_l + C_l,
\end{equation*}
see~\eqref{eqn:mollified_kernel_splitting--l>=0}.
Then we will show that each of the operators $A_l$, $B_l^*$ and
$C_l^*$ is dominated by certain series of ring domain operators, which
are in turn estimated using the main result on ring domain operators,
theorem~\ref{thm:ring_domain_operator}.
In this manner we gain
inequality~\eqref{eqn:mollified_operator_estimate--l>=0-2}, which
reads
\begin{equation}\label{eqn:mollified_operator_estimate--l>=0-1}
  \|P_l\, :\, L_X^p \rightarrow L_X^p \|
  \leq C\, 2^{-l(1 - \frac{1}{\cal T(L_X^p)})},
\end{equation}
where the constant $C$ depends only on $n$, $p$, $X$ and
$\cal T(L_X^p)$.

\bigskip

Using identity \eqref{eqn:kernel--figiel_expansion_simplified} and
dropping the superscripts, we rewrite
equality~\eqref{eqn:mollified_operator}
\begin{align}
  \big( P_l u \big) (x) & = \gint{y}{K_l(x,y)\, u(y)},
  \label{eqn:mollified_operator--kernel_representation}
\intertext{where}
  K_l(x,y)
    & = \sum_{Q,M \in \dcubes}
    \langle f_{Q,l}, h_M \rangle\,
    h_Q(x)\, h_M(y)\, |Q|^{-1} |M|^{-1}.
    \label{eqn:mollified_operator--kernel}
\end{align}

It turns out, that the estimates for the coefficients
$\langle f_{Q,l}, h_M \rangle$ are essentially determined by the
ratio of the diameters of the cubes $Q$ and $M$.
\begin{enumerate}
\item If $\diam(Q) \leq \diam(M)$, using
  $|D_l(Q)| \lesssim 2^{-l}\, |Q|$ and the boundedness of $f_{Q,l}$
  and $h_M$ implies
  \begin{equation}\label{eqn:mollified_haarfun_coeff--l>=0_case_1}
    |\langle f_{Q,l}, h_M \rangle|
    \lesssim 2^{-l}\, |Q|,
  \end{equation}
\item if $2^{-l}\, \diam(Q) \leq \diam(M) < \diam(Q)$, then the
  measure estimate
  \begin{equation*}
    |D_l(Q) \isect M| \lesssim 2^{-l}\, \diam(Q)\,(\diam(M))^{n-1}
  \end{equation*}
  together with inequality~\eqref{eqn:mollified_haarfun--1} yields
  \begin{equation}\label{eqn:mollified_haarfun_coeff--l>=0_case_2}
    |\langle f_{Q,l}, h_M \rangle|
    \lesssim 2^{-l}\, \diam(Q)\,(\diam(M))^{n-1},
  \end{equation}
\item if $\diam(M) < 2^{-l}\, \diam(Q)$, then
  \begin{equation}\label{eqn:mollified_haarfun_coeff--l>=0_case_3}
    |\langle f_{Q,l}, h_M \rangle|
    \lesssim 2^l\, \frac{\diam(M)}{\diam(Q)}\, |M|,
  \end{equation}
  when considering $\lip(f_{Q,l})$ and $\int h_M = 0$ in
  inequality~\eqref{eqn:mollified_haarfun--1}.
\end{enumerate}

Taking a closer look at the case $\diam(Q) \leq \diam(M)$, we observe
that the coefficient $\langle f_{Q,l}, h_M \rangle$ vanishes, if the
support of $f_{Q,l}$ is contained in a set where $h_M$ is constant.
More precisely, let $\{M_i\}_{1 \leq i \leq 2^n}$ be the immediate
dyadic successors of $M$, then if
\begin{equation*}
  \supp f_{Q,l} \subset M_i,
\end{equation*}
for an $1 \leq i \leq 2^n$, we certainly have
\begin{equation*}
  \langle f_{Q,l}, h_M \rangle = 0.
\end{equation*}

\bigskip

Now we focus on estimating the operators $P_l$, with kernel
representation~\eqref{eqn:mollified_operator--kernel}, that was
\begin{equation*}
  K_l(x,y)
    = \sum_{Q,M \in \dcubes}
    \langle f_{Q,l}, h_M \rangle\,
    h_Q(x)\, h_M(y)\, |Q|^{-1} |M|^{-1}.
\end{equation*}
The different behaviour of the
estimates~\eqref{eqn:mollified_haarfun_coeff--l>=0_case_1}~to~\eqref{eqn:mollified_haarfun_coeff--l>=0_case_3}
for the coefficients $f_{Q,l}$ naturally suggests to rearrange the
series in $K_l$ according to the ratio of the diameters of $Q$ and
$M$. 
So we split the set of all pairs of dyadic cubes
$\dcubes\times\dcubes$ in
\begin{align*}
  \cal A_l & = \big\{ (Q,M)\, :\, \diam(Q) \leq \diam(M) \big\},\\
  \cal B_l & = \big\{
    (Q,M)\, :\, 2^{-l}\, \diam(Q) \leq \diam(M) < \diam(Q)
  \big\},\\
  \cal C_l & = \big\{
    (Q,M)\, :\, \diam(M) < 2^{-l}\, \diam(Q)
  \big\},
\end{align*}
and define associated kernels
\begin{equation}
  \begin{aligned}
    A_l(x,y) & = \sum_{(Q,M) \in \cal A_l}
      \langle f_{Q,l}, h_M \rangle\,
      h_Q(x)\, h_M(y)\, |Q|^{-1} |M|^{-1},\\
    B_l(x,y) & = \sum_{(Q,M) \in \cal B_l}
      \langle f_{Q,l}, h_M \rangle\,
      h_Q(x)\, h_M(y)\, |Q|^{-1} |M|^{-1},\\
    C_l(x,y) & = \sum_{(Q,M) \in \cal C_l}
      \langle f_{Q,l}, h_M \rangle\,
      h_Q(x)\, h_M(y)\, |Q|^{-1} |M|^{-1}.
  \end{aligned}
  \label{eqn:mollified_kernel_splitting--l>=0}
\end{equation}

In the following reduction steps for any of the operators $A_l$, $B_l$
and $C_l$ we will decompose each operator or its adjoint into a series
of ring domain operators.

\bigskip

\paragraph{Reduction for $A_l$}\label{p:reduction_A}\hfill

In this case the cube $M$ can be bigger than $Q$. We recall it was
mentioned subordinate to
inequality~\eqref{eqn:mollified_haarfun_coeff--l>=0_case_1}, that the
coefficients $\langle f_{Q,l}, h_M \rangle$ vanish if $h_M$ ist
constant on the support of $f_{Q,l}$.
This setting is illustrated in
figure~\vref{pic:ring_domain--contained_in_cube}.
\begin{figure}[bt]
  \centering
  \includegraphics{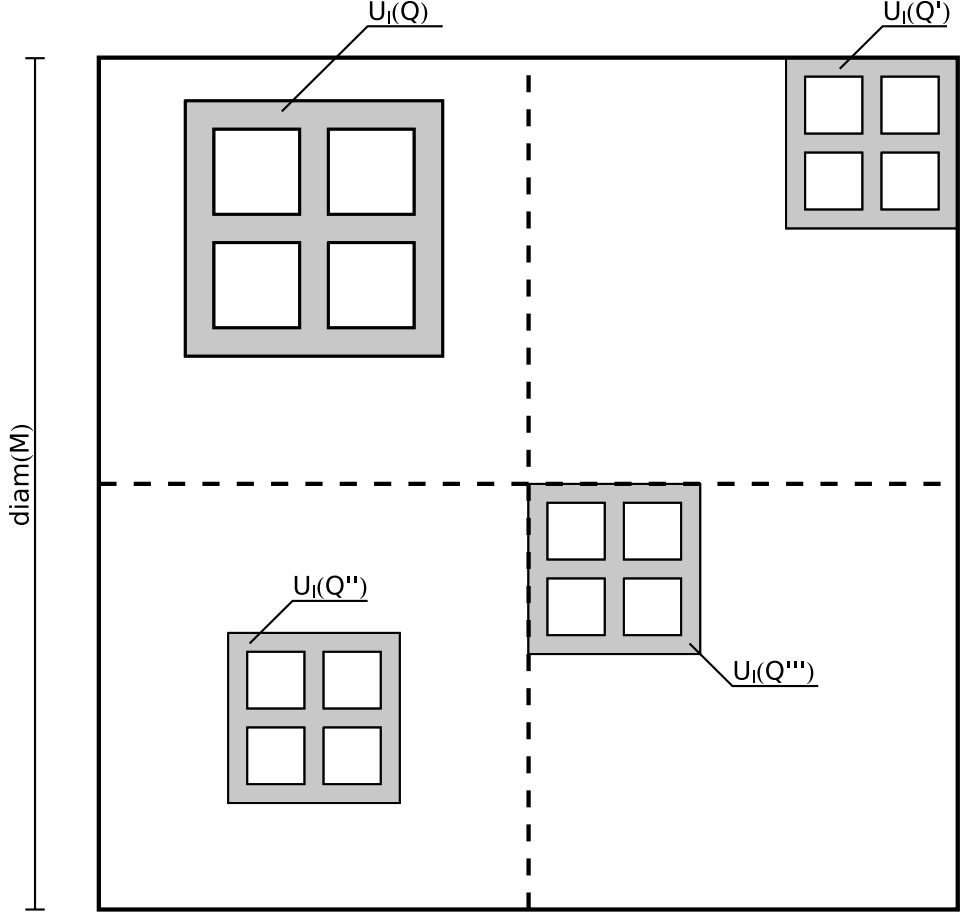}
  \caption{The ring domains (shaded)
    $U_l(Q)$, $U_l(Q')$, $U_l(Q'')$, $U_l(Q''')$ contained in sets of
    constancy of the Haar function $h_M$ (bold continous and dashed
    lines).}
  \label{pic:ring_domain--contained_in_cube}
\end{figure}

\bigskip

At first we parametrize the double series according to $\lambda$,
where $\diam(Q) = 2^{-\lambda}\, \diam(M)$,
\begin{align*}
   A_l u
  & = \sum_{\lambda = 0}^\infty
    \sum_{\substack{Q,M \in \dcubes:\\
          \sidelength(Q) = 2^{-\lambda} \sidelength(M)}}
      \langle f_{Q,l}, h_M \rangle\,
      h_Q\, u_M\, |Q|^{-1} |M|^{-1}\\
  & = \sum_{\lambda = 0}^\infty
      A_{l,\lambda}\, u.
\end{align*}

Observe that with the ratio
\begin{equation*}
  \diam(Q) = 2^{-\lambda}\, \diam(M)
\end{equation*}
fixed and recalling definition~\eqref{eqn:dyadic_support} we have
\begin{equation*}
  \{Q\, :\, \langle f_{Q,l}, h_M \rangle \neq 0\}
  \subset \{ Q\, :\, Q \isect D_\lambda(M) \neq \emptyset \}
  = \cal U_\lambda(M).
\end{equation*}
Using this fact one has the identity
\begin{equation*}
  A_{l,\lambda}\, u
  =  \sum_{M \in \dcubes} u_M\, |M|^{-1}
    \sum_{Q \in \cal U_\lambda(M)}
    \langle f_{Q,l}, h_M \rangle\, |Q|^{-1}\, h_Q,
\end{equation*}
hence glancing at \eqref{eqn:mollified_haarfun_coeff--l>=0_case_1},
utilizing the \umd--property and Kahane's contraction principle, we
obtain
\begin{align*}
  \| A_{l,\lambda} u \|_{L_X^p}
  & \lesssim 2^{-l}\, \big\|
      \sum_{M \in \dcubes} u_M\, |M|^{-1}
      \sum_{Q \in \cal U_\lambda(M)} h_Q
    \big\|_{L_X^p}\\
  & = 2^{-l}\, \big\|
    \sum_{M \in \dcubes} u_M\, g_{M,\lambda}\, |M|^{-1}
    \big\|_{L_X^p}.
\end{align*}
Applying the triangle inequality, using the above estimate for
$A_{l,\lambda}$, considering the definition of the ring domain
operator~\eqref{eqn:ring_domain_operator}, and invoking
theorem~\ref{thm:ring_domain_operator} yields
\begin{equation*}
  \| A_l u \|_{L_X^p}
  \leq \sum_{\lambda = 0}^\infty \| A_{l,\lambda} u \|_{L_X^p}
  \lesssim 2^{-l}\, \sum_{\lambda=0}^\infty \|S_\lambda u\|_{L_X^p}
  \lesssim 2^{-l}\, \sum_{\lambda=0}^\infty
    2^{-\lambda/\cal C(L_X^p)}\, \|u\|_{L_X^p}.
\end{equation*}
Evaluating the geometric series we attain the estimate
\begin{equation}\label{eqn:mollified_operator_estimate--l>=0_case_1}
  \| A_l u \|_{L_X^p}
  \leq C\, 2^{-l}\, \|u\|_{L_X^p},
\end{equation}
where the constant $C$ depends on $n$, $p$, $X$ and $\cal C(L_X^p)$.

\bigskip

Being aware that with $\lambda \geq 0$ fixed, the collections
$\cal U_\lambda(M)$ are not disjoint as $M$ ranges over $\dcubes$ but
the overlap is bounded by a constant depending solely on the dimension
$n$ and the constant appearing in the definition of $D_\lambda(Q)$, we
could have partitioned $\dcubes$ in a constant number of sets,
generically denoted by $\scr B \subset \dcubes$, such that the
$\cal U_\lambda(M)$ would not have interfered with each other in the
first place. Then one can repeat the argument above, with $\dcubes$
replaced by one of the collections $\scr B$.

\bigskip

\paragraph{Reduction for $B_l$}\label{p:reduction_B}\hfill

This setting is visualised in
figure~\vref{pic:ring_domain--cubes_in_between}.
\begin{figure}[bt]
  \centering
  \includegraphics{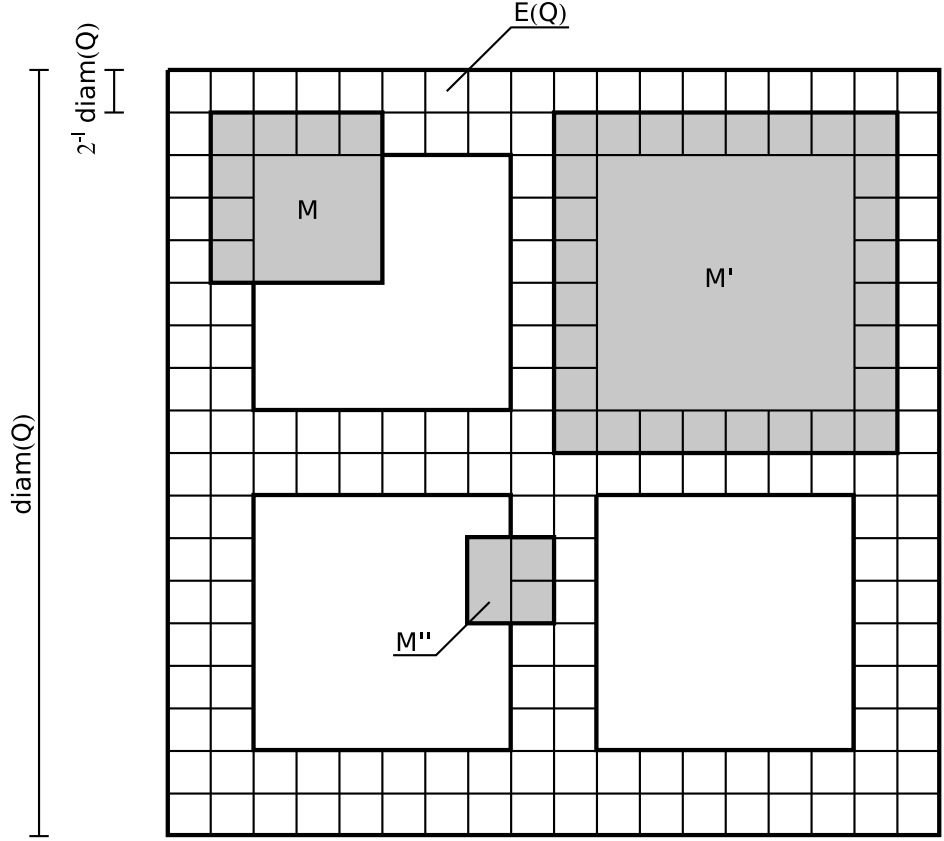}
  \caption{The cubes $M$, $M'$ and $M''$ (shaded) intersecting the
    cubes $E(Q)$ of $\cal U_l(Q)$ (thin lines).}
  \label{pic:ring_domain--cubes_in_between}
\end{figure}
Note that the cubes $M$ are now smaller than $Q$, but bigger than the
building blocks of the ring domain $\cal U_l(Q)$, so
$2^{-l} \diam(Q) \leq \diam(M) < \diam(Q)$.
We may use inequality~\eqref{eqn:mollified_haarfun_coeff--l>=0_case_2}
for estimating $\langle f_{Q,l}, h_M \rangle$.

This time we prefer to analyze $B_l^*$, certainly with respect to the
norm $\|\cdot\|_{L_Y^q}$, where $Y = X^*$ and
$\frac{1}{p} + \frac{1}{q} = 1$. As before we rearrange the series to
see
\begin{align*}
   B_l^* u
   & = \sum_{\lambda = 1}^l
     \sum_{\substack{Q,M \in \dcubes:\\
         \sidelength(M) = 2^{-\lambda} \sidelength(Q)}}
     \langle f_{Q,l}, h_M \rangle\, |M|^{-1}
     h_M\, u_Q\, |Q|^{-1}\\
  & = \sum_{\lambda = 0}^\infty B_{l,\lambda}^*\, u.
\end{align*}

When restricted to the fixed ratio
\begin{equation*}
  \diam(M) = 2^{-\lambda}\, \diam(Q),
\end{equation*}
note that
\begin{equation*}
  \{ M\, :\, \langle f_{Q,l}, h_M \rangle \neq 0 \}
  \subset \{ M\, :\, M \isect D_l(Q) \neq \emptyset \}
  = \cal U_\lambda(Q),
\end{equation*}
so we can rewrite $B_{l,\lambda}^* u$ as follows:
\begin{equation*}
  B_{l,\lambda}^* u
  = \sum_{Q \in \dcubes} u_Q\, |Q|^{-1}
      \sum_{M \in \cal U_\lambda(Q)}
        \langle f_{Q,l}, h_M \rangle\, |M|^{-1} h_M.
\end{equation*}
Taking the norm, utilizing the $\umd$--property and applying Kahane's
contraction principle
to~\eqref{eqn:mollified_haarfun_coeff--l>=0_case_2} yields the
estimate
\begin{align*}
  \| B_{l,\lambda}^* u \|_{L_Y^q}
  & \lesssim 2^{-l}\,
    \big\|
      \sum_{Q \in \dcubes} u_Q\, |Q|^{-1}
      \sum_{M \in \cal U_\lambda(Q)} h_M
    \big\|_{L_Y^q}\\
  & = 2^{-l}\,
    \big\|
      \sum_{Q \in \dcubes} u_Q\, g_{Q,\lambda} |Q|^{-1}
    \big\|_{L_Y^q}.
\end{align*}
In view of theorem~\ref{thm:ring_domain_operator} one proceeds
\begin{equation*}
  \| B_l^* u \|_{L_Y^q}
  \leq \sum_{\lambda = 0}^\infty \| B_{l,\lambda}^*\, u \|_{L_Y^q}
  \lesssim 2^{-l} \sum_{\lambda = 1}^l 2^\lambda\,
    \big\| S_\lambda u \big\|_{L_Y^q}
    \lesssim 2^{-l} \sum_{\lambda = 1}^l
    2^{\lambda(1-1/\cal C(L_Y^q))}\,
    \big\| u \big\|_{L_Y^q},
\end{equation*}
to conclude this case, retaining
\begin{equation}\label{eqn:mollified_operator_estimate--l>=0_case_2}
  \| B_l^* u \|_{L_Y^q}
  \leq C\, 2^{-l/\cal C(L_Y^q)}
    \big\| u \big\|_{L_Y^q},
\end{equation}
where the constant $C$ depends on $n$, $q$ and $Y$.

\bigskip

\paragraph{Reduction for $C_l$}\label{p:reduction_C}\hfill

We may think of the cube $M$ being much smaller than $Q$, even smaller
than the building blocks of the ring domain $\cal U_l(Q)$, and we have
inequality~\eqref{eqn:mollified_haarfun_coeff--l>=0_case_3} at our
disposal. This is visualised in
figure~\vref{pic:ring_domain--cubes_contained_in_covering}.
\begin{figure}[bt]
  \centering
  \includegraphics{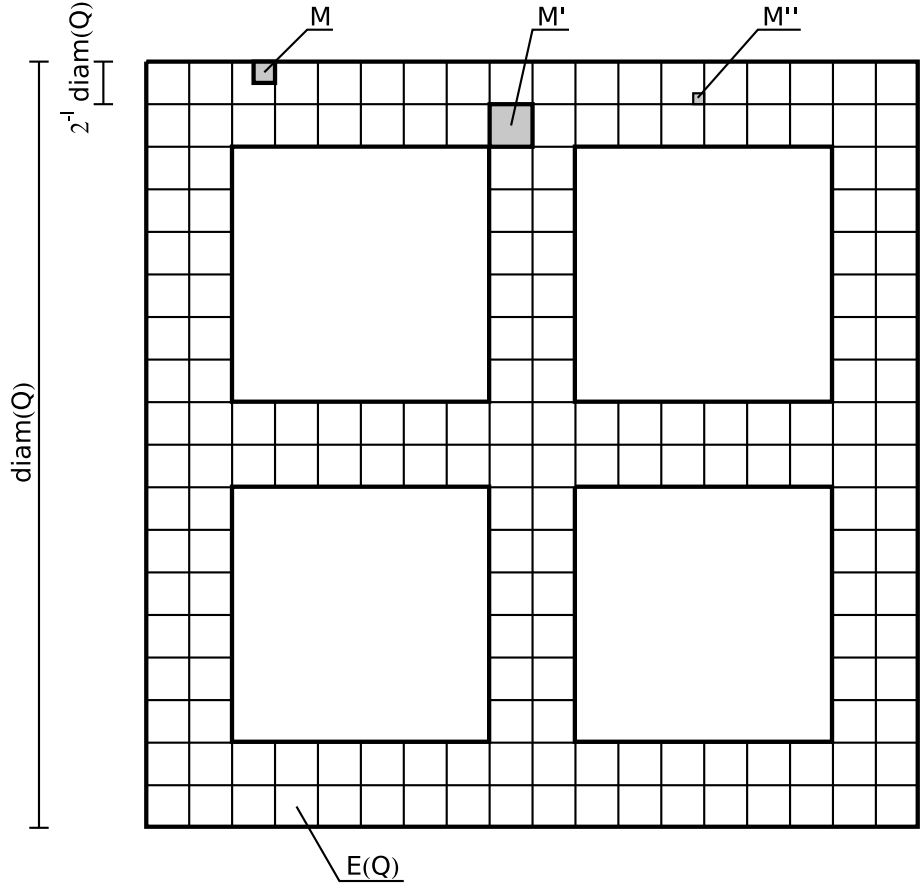}
  \caption{The cubes $M$, $M'$ and $M''$ (shaded) contained in the
    cover $\cal U_l(Q)$ (thin lines).}
  \label{pic:ring_domain--cubes_contained_in_covering}
\end{figure}

As in the preceeding case we aim at estimating the adjoint operator
$C_l^*$; so with $Y=X^*$ and $\frac{1}{p} + \frac{1}{q} = 1$ the
usual parametrization leads to
\begin{align*}
  C_l^* u
  & = \sum_{\lambda = l+1}^\infty
    \sum_{\substack{Q,M \in \dcubes:\\
        \sidelength(M) = 2^{-\lambda} \sidelength(Q)}}
      \langle f_{Q,l}, h_M \rangle\, |M|^{-1}
      h_M\, u_Q\, |Q|^{-1}\\
  & = \sum_{\lambda = l+1}^\infty C_{l,\lambda}^* u.
\end{align*}

Under the restriction of $\diam(M) = 2^{-\lambda}\, \diam(Q)$ holds
that
\begin{equation*}
  \{ M\, :\, \langle f_{Q,l}, h_M \rangle \neq 0 \}
  \subset \{ M\, :\, M \isect D_l(Q) \neq \emptyset \}
  \approx \cal U_l(Q).
\end{equation*}
Note that the last equality is not true algebraically, indicated by
''$\approx$''. This notation is justified by the $\umd$--property and
Kahane's contraction principle, which enables us to exchange zero mean
Haar functions, as long as their supports are preserved.

We proceed by applying essentially the same steps as supplied before,
now having estimate~\eqref{eqn:mollified_haarfun_coeff--l>=0_case_3}
at hand.
Using the $\umd$--property and Kahane's contraction principle
to~\eqref{eqn:mollified_haarfun_coeff--l>=0_case_3} and
\begin{equation*}
\big| \sum_{
  \substack{
    M \in \dcubes\, :\, M \isect D_l(Q) \neq \emptyset\\
    \sidelength(M) = 2^{-\lambda} \sidelength(Q)
  }
} h_M \big|
\leq \big| \sum_{M \in \cal U_l(Q)} h_M \big| = |g_{Q,l}|,
\end{equation*}
we gain
\begin{equation*}
  \| C_{l,\lambda}^* u \|_{L_Y^q}
  \lesssim 2^l\, 2^{-\lambda}
    \big\|
      \sum_{Q \in \dcubes} u_Q\, g_{Q,l}\, |Q|^{-1}
    \big\|_{L_Y^q},
\end{equation*}
thus, the triangle inequality and the above estimate for
$C_{l,\lambda}^*$ yield
\begin{equation*}
  \| C_l^* u \|_{L_Y^q}
  \lesssim \big\| S_l u \big\|_{L_Y^q}.
\end{equation*}
Finally, theorem~\ref{thm:ring_domain_operator} yields
\begin{equation}\label{eqn:mollified_operator_estimate--l>=0_case_3}
  \| C_l^* u \|_{L_Y^q}
  \leq C\, 2^{-l/\cal C(L_Y^q)},
\end{equation}
where the constant $C$ depends only on $n$, $q$ and $Y$.

\bigskip

\paragraph{Summary}\hfill

We combine the
inequalities~\eqref{eqn:mollified_operator_estimate--l>=0_case_1},
\eqref{eqn:mollified_operator_estimate--l>=0_case_2},
\eqref{eqn:mollified_operator_estimate--l>=0_case_3}, exploit
that for $Y = X^*$ and $\frac{1}{p} + \frac{1}{q} = 1$ holds
\begin{equation*}
  (L_X^p)^* = L_Y^q
  \qquad \text{and} \qquad
  \frac{1}{\cal T(L_X^p)} + \frac{1}{\cal (L_Y^q)} = 1,
\end{equation*}
to obtain
\begin{equation}\label{eqn:mollified_operator_estimate--l>=0-2}
  \|P_l\, :\, L_X^p \rightarrow L_X^p \|
  \leq C\, 2^{-l(1 - \frac{1}{\cal T(L_X^p)})},
\end{equation}
where $L_X^p$ has type $\cal T(L_X^p)$ and the constant
$C$ depends only on $n$, $p$, $X$, particularly on $\cal T(L_X^p)$ and
$\cal C(L_X^p)$.

\bigskip

Very seperate reasons, that was the special shape of the support of
$f_{Q,l}$ on the one hand, and the constancy of the Haar function
$h_M$ exploiting the zero mean of $f_{Q,l}$ on the other hand, enabled
us to reduce the estimates for $P_l$ to ring domain operators
\begin{equation*}
  S_\lambda u = \sum_{Q \in \dcubes} u_Q\, g_{Q,\lambda}\, |Q|^{-1},
\end{equation*}
where
\begin{equation*}
  g_{Q,\lambda} = \sum_{E \in \cal U_\lambda(Q)} h_E,
  \qquad \text{and} \qquad
  u = \sum_{Q \in \dcubes} u_Q\, h_Q\, |Q|^{-1}.
\end{equation*}

\bigskip

\subsubsection{Estimates for  $P_l^{(\varepsilon)}$, $l < 0$}\label{sss:decomposition_estimates--l<0}\hfill

We want to find estimates for the remaining sum
\begin{equation*}
  P_- = \sum_{l < 0} P_l.
\end{equation*}
The argument is analogously to the case $l \geq 0$ completed
in~\ref{sss:decomposition_estimates--l>=0}.
The splitting of $\dcubes \times \dcubes$ will be according to the
behaviour of $\langle f_Q, h_M \rangle$ in
\eqref{eqn:mollified_haarfun_coeff--l=0_case_1}~and~\eqref{eqn:mollified_haarfun_coeff--l=0_case_2},
inducing the decomposition of $P_-$.
The functions $f_Q$ are defined beneath.

\bigskip

Dropping all superscripts we issue
representation~\eqref{eqn:mollified_kernel}~\vpageref{eqn:mollified_kernel}
for the kernel of $P_l$
\begin{equation*}
  K_l(x,y)
  = \sum_{Q \in \dcubes} h_Q(x)\, f_{Q,l}(y)\, |Q|^{-1},
\end{equation*}
and recall that for $Q \in \dcubes_j$
\begin{equation*}
  f_{Q,l}
  = \Delta_{j+l} h_Q
  = h_Q * d_{j+l}
  = h_Q * (b_{j+l+1} - b_{j+l}).
\end{equation*}
Taking the sum over $l < 0$ yields
\begin{equation*}
  \sum_{l < 0} f_{Q,l} = h_Q * d_j
\end{equation*}
since
$\lim_{l \rightarrow \infty} b_{j+l} \xrightarrow{\text{p.w.}} 0$,
and so we define the mollified Haar functions
\begin{equation*}
  f_Q = h_Q * d_j, \qquad \text{for all $Q \in \dcubes_j$}.
\end{equation*}
For the properties of the mollifier $d_j$ one might want to take a
look at~\eqref{eqn:decomposition:vanishing_moments}
and~\eqref{eqn:decomposition:convolution}.

In this way we obtain the kernel $K_-(x,y)$ of the operator $P_-$
\begin{equation*}
  K_-(x,y)
  = \sum_{l < 0} K_l(x,y)
  = \sum_{Q \in \dcubes} h_Q(x)\, f_Q(y)\, |Q|^{-1}.
\end{equation*}
In
\eqref{eqn:mollified_haarfun--1}~\vpageref{eqn:mollified_haarfun--1}
we observed that for $l = 0$ there exists a $C > 0$ so that for
$Q \in \dcubes$
\begin{equation}
  \begin{aligned}
    \gint{x}{f_Q(x)} & = 0,
    & \supp f_Q & \subset C \cdot Q,\\
    |f_Q| & \leq C,
    & \lip(f_Q) & \leq C\, (\diam(Q))^{-1}.
  \end{aligned}
  \label{eqn:mollified_haarfun--l=0}
\end{equation}

Figiel's expansion, the $\umd$--property and Kahane's contraction
principle yields the following generic form of the kernel
\begin{equation}\label{eqn:sir_kernel--figiel_expansion--l=0}
  K_-(x,y) = 
  \sum_{M,Q \in \dcubes}
  \langle f_Q, h_M \rangle\,
  |M|^{-1} |Q|^{-1} h_Q(x)\, h_M(y).
\end{equation}

Again, we need to analyze the properties of the coefficients
$\langle f_Q, h_M \rangle$. As in the preceeding cases, the estimates
strongly depend on the ratio of the diameters of $Q$ and $M$.
\begin{enumerate}
\item If $\diam(M) \leq \diam(Q)$, we make use of
  \begin{equation*}
    \lip(f_Q)
      \leq C\, (\diam(Q))^{-1},
  \end{equation*}
  according to
  \eqref{eqn:mollified_haarfun--l=0}
  and discover
  \begin{equation}\label{eqn:mollified_haarfun_coeff--l=0_case_1}
    |\langle f_Q, h_M \rangle|
    \lesssim (\diam(Q))^{-1}\, (\diam(M))^{n+1},
  \end{equation}
\item while if $\diam(M) > \diam(Q)$, one can exploit
  \begin{equation*}
    |f_Q| \leq C
    \qquad \text{and} \qquad
    \supp f_Q \subset C \cdot Q
  \end{equation*}
  to obtain
  \begin{equation}\label{eqn:mollified_haarfun_coeff--l=0_case_2}
    |\langle f_Q, h_M \rangle| \lesssim |Q|.
  \end{equation}
\end{enumerate}

The coefficient $\langle f_Q, h_M \rangle$ vanishes if the
support of $f_Q$ is contained in a set where $h_M$ is constant.
Precisely, let $\{M_i\}_{1 \leq i \leq 2^n}$ be the immediate dyadic
successors of $M$, then if
\begin{equation*}
  \supp f_Q \subset M_i,
\end{equation*}
for an $1 \leq i \leq 2^n$, we have
\begin{equation*}
  \langle f_Q, h_M \rangle = 0.
\end{equation*}

Hence, for $\diam(M) > \diam(Q)$ the cubes $Q$ for which
$\langle f_Q, h_M \rangle \neq 0$ cluster in the vicinity of $D(M)$,
the set of $h_M$'s discontinuities.

\bigskip

We start the analysis of the Operators $P_-$ using the
representation~\eqref{eqn:sir_kernel--figiel_expansion--l=0}, that is
\begin{align*}
  \big( P_- u \big) (x) & = \gint{y}{K_-(x,y)\, u(y)},
\intertext{and}
  K_-(x,y)
    & = \sum_{Q,M \in \dcubes}
    \langle f_Q, h_M \rangle\,
    h_Q(x)\, h_M(y)\, |Q|^{-1} |M|^{-1}.
\end{align*}
Driven by \eqref{eqn:mollified_haarfun_coeff--l=0_case_1} and
\eqref{eqn:mollified_haarfun_coeff--l=0_case_2}, we split the set of
all pairs of dyadic cubes $\dcubes\times\dcubes$ in
\begin{align*}
  \cal A_- & = \big\{ (Q,M)\, :\, \diam(M) \leq \diam(Q) \big\},\\
  \cal B_- & = \big\{ (Q,M)\, :\, \diam(M) > \diam(Q) \big\},
\end{align*}
and define the associated kernels
\begin{align*}
  A_-(x,y) & = \sum_{(Q,M) \in \cal A_-}
    \langle f_Q, h_M \rangle\,
    h_Q(x)\, h_M(y)\, |Q|^{-1} |M|^{-1},\\
  B_-(x,y) & = \sum_{(Q,M) \in \cal B_-}
    \langle f_Q, h_M \rangle\,
    h_Q(x)\, h_M(y)\, |Q|^{-1} |M|^{-1},
\end{align*}
accordingly.

\bigskip

\paragraph{Estimates for $A_-$}\hfill

In this case the size of cube $M$ cannot exceed that of $Q$, so we
may use inequality~\eqref{eqn:mollified_haarfun_coeff--l=0_case_1}.
We rather want to estimate $A_-^*$ than $A_-$ itself, therefore
$Y = X^*$ and $\frac{1}{p} + \frac{1}{q} = 1$.
Rearranging the series in $A_-^*$ according to the ratio of the
diameters of $Q$ and $M$ yields
\begin{align*}
  A_-^* u
  & = \sum_{\lambda = 0}^\infty
    \sum_{\substack{Q,M \in \dcubes:\\
        \sidelength(M) = 2^{-\lambda} \sidelength(Q)}}
    \langle f_Q, h_M \rangle\, u_Q\, |Q|^{-1} h_M\, |M|^{-1}\\
  & = \sum_{\lambda = 0}^\infty A_{-,\lambda}^* u.
\end{align*}
Taking the norm and applying the triangle inequality to the first sum
\begin{equation*}
  \| A_-^* u \|_{L_Y^q}
  \leq \sum_{\lambda = 0}^\infty
    \big\| A_{-,\lambda}^* u \|_{L_Y^q},
\end{equation*}
necessitates to estimate $A_{-,\lambda}^*$.
Utilizing the \umd--property and Kahane's contraction
principle~\eqref{eqn:kahanes_contraction_principle} applied
to~\eqref{eqn:mollified_haarfun_coeff--l=0_case_1}, we infer
\begin{equation*}
  \big\| A_{-,\lambda}^* u \big\|_{L_Y^q}
  \lesssim 2^{-\lambda}
    \big\| \sum_{Q \in \dcubes}
      \sum_{\substack{
          \sidelength(M) = 2^{-\lambda} \sidelength(Q)\\
          M \isect (C \cdot Q) \neq \emptyset
        }
      }
      u_Q\, |Q|^{-1} h_M
    \big\|_{L_Y^q}.
\end{equation*}
For every $Q \in \dcubes$ we issue Kahane's contraction principle on
\begin{equation*}
  \big| \sum_{
    \substack{
      \sidelength(M) = 2^{-\lambda} \sidelength(Q)\\
      M \isect (C \cdot Q) \neq \emptyset
    }
  }
  h_M
  \big|
  \leq |h_Q|,
\end{equation*}
keeping in mind that we would actually need a constant number of
Figiel shifts of $h_Q$ to cover the whole support of this sum.
Nevertheless, the bound~\eqref{eqn:shift_operator_1_estimate} allows
us to estimate
\begin{align*}
  \big\| A_{-,\lambda}^* u \big\|_{L_Y^q} 
  & \lesssim 2^{-\lambda}\, \big\|
    \sum_{Q \in \dcubes} u_Q\, |Q|^{-1}
    \sum_{\substack{
        \sidelength(M) = 2^{-\lambda} \sidelength(Q)\\
        M \isect (C \cdot Q) \neq \emptyset
      }
    }
    h_M
  \big\|_{L_Y^q}\\
  & \lesssim 2^{-\lambda}\, \big\| u \big\|_{L_Y^q}.
\end{align*}
To conclude, we string together our estimates, yielding
\begin{equation}\label{eqn:mollified_operator_estimate--l=0_case_1}
  \| A_-^* u \|_{L_X^p}
  \leq C\, \|u\|_{L_X^p},
\end{equation}
where the constant $C$ depends on $n$, $q$ and $Y$.

\bigskip

\paragraph{Estimates for $B_-$}\hfill

In this setting the size of $M$ does exceed $Q$.

With the usual parametrization of $B_-$ we have
\begin{align*}
  B_- u
  & = \sum_{\lambda = 1}^\infty
    \sum_{\substack{Q,M \in \dcubes:\\\diam(Q) = 2^{-\lambda} \diam(M)}}
    \langle f_Q, h_M \rangle\,h_Q\, |Q|^{-1} u_M\, |M|^{-1}\\
  & = \sum_{\lambda = 1}^\infty B_{-,\lambda} u.
\end{align*}
Restricted to cubes $Q,M \in \dcubes$ with
\begin{equation*}
  2^{-\lambda} \diam(M) = \diam(Q)
\end{equation*}
one can see that for all $M$ the following holds true:
\begin{equation*}
  \{Q\, :\, \langle f_Q, h_M \rangle \neq 0 \}
  \subset \{Q\, :\, (C \cdot Q) \isect D(Q) \neq \emptyset\}
  \subset \cal U_\lambda(M).
\end{equation*}
Successively using the \umd--property, Kahane's contraction principle
applied to~\eqref{eqn:mollified_haarfun_coeff--l=0_case_2} and the
inclusion above we obtain
\begin{align*}
  \big\| B_{-,\lambda} u \big\|_{L_X^p}
  & \lesssim \big\|
  \sum_{M \in \dcubes} u_M\, |M|^{-1}\,
    \sum_{Q \in \cal U_\lambda(M)} h_Q
  \big\|_{L_X^p}\\
  & = \big\|
  \sum_{M \in \dcubes} u_M\, g_{M,\lambda}\, |M|^{-1}
  \big\|_{L_X^p}.
\end{align*}
The main result on ring domain operators
theorem~\ref{thm:ring_domain_operator} yields
\begin{equation*}
  \| B_{-,\lambda} u \|_{L_X^p}
  \lesssim \| S_\lambda u \|_{L_X^p}
  \lesssim 2^{-\lambda/\cal C(L_X^p)}\, \|u\|_{L_X^p},
\end{equation*}
hence merging our inequalities we attain
\begin{equation}\label{eqn:mollified_operator_estimate--l=0_case_2}
  \|B_- u\|_{L_X^p} \leq C\, \|u\|_{L_X^p},
\end{equation}
where the constant $C$ depends on $n$, $p$, $X$ and $\cal C(L_X^p)$.

\bigskip

\paragraph{Summary}\hfill

Inequality~\eqref{eqn:mollified_operator_estimate--l=0_case_1}
and~\eqref{eqn:mollified_operator_estimate--l=0_case_2} together imply
the boundedness for the mollified operator $P_-$
\begin{equation}\label{eqn:mollified_operator_estimate--l=0}
  \|P_-\, :\, L_X^p \rightarrow L_X^p \| \leq C,
\end{equation}
where the constant $C$ depends on $n$, $p$, $X$, particularly on
$\cal T(L_X^p)$ and $\cal C(L_X^p)$.

In the case $l \leq 0$, the shape of the support of the mollified Haar
function $f_{Q,l}$ is not a ring domain, opposed to the case $l \geq
0$. So we cannot expect to reduce our estimates to ring domain
operators in cases where the shape of the support of $f_{Q,l}$ is
crucial. Revisiting the reduction to ring domain operators for
$l \geq 0$, it is clear that the reduction to ring domain operators is
feasible for the operator $B_-$, since we can still exploit the zero
mean of $f_{Q,l}$ on sets where $h_M$ is constant.

\subsection{Estimates for $P_l^{(\varepsilon)} R_{i_0}^{-1}$}\label{ss:decomposition_riesz_estimates}\hfill

In this brief section we will establish estimates for
$P_l^{(\varepsilon)} R_{i_0}^{-1}$, $l \in \bb Z$ by reducing them to
estimates for $P_l^{(\varepsilon)}$. This necessitates that
$\big( R_{i_0}^{-1} \big)^*$ maps the mollified Haar functions
$f_{Q,l}^{(\varepsilon)}$ to functions $k_{Q,l}^{(\varepsilon)}$
enjoying similar properties.
Due to the algebraic identity~\eqref{eqn:riesz_transform_inverse},
this amounts to controlling the support of the $k_{Q,l}$ (besides
factors depending on $l$).
Assuming $\varepsilon_{i_0} = 1$, one can exploit
\begin{equation*}
  \supp \big( \cond_{i_0} h_Q^{(\varepsilon)} \big)
  \subset Q,
\end{equation*}
provoking the functions $k_{Q,l,i}$ defined
in~\eqref{eqn:mollified_haarfun_riesz} to exhibit the support
conditions asserted in~\eqref{eqn:mollified_haarfun_riesz--1}
and~\eqref{eqn:mollified_haarfun_riesz--2}.

\bigskip

It is a well known fact that one can write the inverse of the Riesz
transform as
\begin{equation}\label{eqn:riesz_transform_inverse}
  R_{i_0}^{-1}
  = R_{i_0} + \sum_{\substack{1 \leq i \leq n\\i \neq i_0}}
    \bb E_{i_0} \partial_i R_i,
\end{equation}
where $\bb E_{i_0}$ denotes integration with respect to the $i_0$--th
variable,
\begin{equation*}
  \bb E_{i_0}f (x)
  = \int_{-\infty}^{x_{i_0}}
    f(x_1,\ldots,x_{i_0-1},s,x_{i_0+1},\ldots,x_n)\, \mathrm ds,
    \qquad x = (x_1,\ldots,x_n).
\end{equation*}

Now we introduce the family of functions
\begin{equation}\label{eqn:mollified_haarfun_riesz}
  k_{Q,l,i}^{(\varepsilon)}
  = \Delta_{j+l} \big( \bb E_{i_0} \partial_i h_Q^{(\varepsilon)} \big),
  \qquad \text{if $Q \in \dcubes_j$},
\end{equation}
and consider
\begin{align*}
  P_l^{(\varepsilon)} R_{i_0}^{-1} u
  & = \sum_{j \in \bb Z} \sum_{Q \in \dcubes_j}
    \big\langle
      R_{i_0} u, \Delta_{j+l}(h_Q^{(\varepsilon)})
    \big\rangle\,
    h_Q^{(\varepsilon)}\, |Q|^{-1}\\
  &\quad + \sum_{\substack{1 \leq i \leq n\\i \neq i_0}}
  \sum_{j \in \bb Z} \sum_{Q \in \dcubes_j}
    \big\langle
      \bb E_{i_0} \partial_i R_i u, \Delta_{j+l}(h_Q^{(\varepsilon)})
    \big\rangle\,
    h_Q^{(\varepsilon)}\, |Q|^{-1}.
\end{align*}
Since the Riesz transforms are continous mappings, it is obvious that
the first sum can be treated as in
section~\ref{ss:decomposition_estimates}.
For the seond sum, we fix a coordinate $i \neq i_0$, rearrange the
operators in the scalar product and use the functions defined in
\eqref{eqn:mollified_haarfun_riesz}, hence
\begin{equation*}
  \sum_{j \in \bb Z} \sum_{Q \in \dcubes_j}
    \langle
      \bb E_{i_0} \partial_i R_i u, \Delta_{j+l}(h_Q^{(\varepsilon)})
    \rangle\,
    h_Q^{(\varepsilon)}\, |Q|^{-1}
  = \sum_{Q \in \dcubes}
    \langle R_i u, k_{Q,l,i}^{(\varepsilon)} \rangle\,
    h_Q^{(\varepsilon)}\, |Q|^{-1}.
\end{equation*}
The continuity of the Riesz transforms $R_i:L_X^p \rightarrow L_X^p$
allows us to estimate the following simpler type of operator
\begin{equation*}
  K_{l,i}^{(\varepsilon)} u
  = \sum_{Q \in \dcubes}
    \langle u, k_{Q,l,i}^{(\varepsilon)} \rangle\,
    h_Q^{(\varepsilon)}\, |Q|^{-1}.
\end{equation*}

In order to estimate $K_{l,i}^{(\varepsilon)}$ we need to analyze the
analytic properties of the functions $k_{Q,l,i}^{(\varepsilon)}$.
If $l \geq 0$, then
\begin{equation}
  \begin{aligned}
    \gint{x}{k_{Q,l,i}^{(\varepsilon)}(x)} & = 0,
    & \supp k_{Q,l,i}^{(\varepsilon)} & \subset D_l^{(\varepsilon)}(Q),\\
    |k_{Q,l,i}^{(\varepsilon)}| & \leq C\, 2^l,
    & \lip(k_{Q,l,i}^{(\varepsilon)}) & \leq C\, 2^{2l}\, (\diam(Q))^{-1},
  \end{aligned}
  \label{eqn:mollified_haarfun_riesz--1}
\end{equation}
and for $l \leq 0$
\begin{equation}
  \begin{aligned}
    \gint{x}{k_{Q,l,i}^{(\varepsilon)}(x)}
      & = 0,
    & \supp k_{Q,l,i}^{(\varepsilon)}
      & \subset C\, 2^{|l|} Q,\\
    |k_{Q,l,i}^{(\varepsilon)}|
      & \leq C\, 2^{-|l|(n+1)},
    & \lip(k_{Q,l,i}^{(\varepsilon)})
      & \leq C\, 2^{-|l|(n+2)}\, (\diam(Q))^{-1}.
  \end{aligned}
  \label{eqn:mollified_haarfun_riesz--2}
\end{equation}
Note that the above properties of $k_{Q,l,i}^{(\varepsilon)}$
especially depend on the coordinatewise vanishing moments of
$b$~\eqref{eqn:decomposition:vanishing_moments}, introduced by
$\Delta_l$ in equations~\eqref{eqn:decomposition:convolution}
and~\eqref{eqn:dfn:mollified_operator--1}.
Furthermore observe the definition of $k_{Q,l,i}^{(\varepsilon)}$
involves an integration of $h_Q^{(\varepsilon)}$ with respect to the
variable $x_{i_0}$. Now if $\varepsilon_{i_0} = 1$, then
$\cond_{i_0 } h_Q^{(\varepsilon)}$ is compactly supported in $Q$, but
if $\varepsilon_{i_0} = 0$, then
$\supp\big( \cond_{i_0 } h_Q^{(\varepsilon)} \big)$ is unbounded.
This urges the dominating Riesz transform $R_{i_0}$ to act on a
coordinate $x_{i_0}$ for which $P^{(\varepsilon)}$ projects onto zero
mean Haar functions, thus necessitating $\varepsilon_{i_0} = 1$.

If we compare this with the
properties~\eqref{eqn:mollified_haarfun--1} and
\eqref{eqn:mollified_haarfun--2} regarding the functions
$f_{Q,l}^{(\varepsilon)}$, it turns out that the properties coincide
if $l \leq 0$, and that $2^{-l}\, k_{Q,l,i}^{(\varepsilon)}$,
satisfies the same conditions as $f_{Q,l}^{(\varepsilon)}$, if
$l \geq 0$.
Bootstrapping the proofs in section~\ref{ss:decomposition_estimates},
we note that those arguments where solely depending on the analytic
properties~\eqref{eqn:mollified_haarfun--1} and
\eqref{eqn:mollified_haarfun--2} of the functions
$f_{Q,l}^{(\varepsilon)}$.
With regard to~\eqref{eqn:mollified_haarfun_riesz--1}
respectively~\eqref{eqn:mollified_haarfun_riesz--1}, the same proofs
are practicable with the functions $k_{Q,l,i}^{(\varepsilon)}$, if
$l \leq 0$, respectively $2^{-l}\, k_{Q,l,i}^{(\varepsilon)}$, if
$l \geq 0$, replacing $f_{Q,l}$.
Stringing this all together implies the following upper bounds for the
operators $K_{l,i}^{(\varepsilon)}$ and $K_{-,i}^{(\varepsilon)}$,
where
\begin{equation*}
  K_{-,i}^{(\varepsilon)} = \sum_{l \leq 0} K_{l,i}^{(\varepsilon)}.
\end{equation*}

Estimate~\eqref{eqn:mollified_operator_estimate--l=0} implies
\begin{equation*}
  \|K_{-,i}^{(\varepsilon)}\, :\, L_X^p \rightarrow L_X^p \| \leq C,
\end{equation*}
and if $l \geq 0$, then using
estimate~\eqref{eqn:mollified_operator_estimate--l>=0-1} on
$2^{-l}\, K_{l,i}^{(\varepsilon)}$ yields
\begin{equation*}
  \|K_{l,i}^{(\varepsilon)}\, :\, L_X^p \rightarrow L_X^p \|
  \leq C\, 2^{l/\cal T(L_X^p)},
\end{equation*}
where the constant $C$ depends on $n$, $p$, $X$, particularly on
$\cal T(L_X^p)$ and $\cal C(L_X^p)$.

Obviously, the estimate for the operators
$P_l^{(\varepsilon)} R_{i_0}^{-1}$ ist just a constant multiple of the
operator norm of $K_{l,i}^{(\varepsilon)}$, so we summarize:

If $\varepsilon_{i_0} = 1$, then the following inequalities hold true:
\begin{align}
  \|P_-^{(\varepsilon)} R_{i_0}^{-1}\, :\, L_X^p \rightarrow L_X^p \|
  \leq C,
  \label{eqn:mollified_operator_riesz_estimate--l<=0}
\intertext{and for all $l \geq 0$}
  \|P_l^{(\varepsilon)} R_{i_0}^{-1}\, :\, L_X^p \rightarrow L_X^p \|
  \leq C\, 2^{l/\cal T(L_X^p)},
\label{eqn:mollified_operator_riesz_estimate--l>=0}
\end{align}
where the constant $C$ depends merely on $n$, $p$, $X$,
$\cal T(L_X^p)$ and $\cal C(L_X^p)$.

\newpage
\bibliographystyle{alpha}
\bibliography{bibliography}
\nocite{*}
\end{document}